\documentclass[11pt]{article}
\usepackage{graphicx}
\usepackage{amssymb}
\usepackage{amsmath}
\usepackage{amsthm}
\usepackage{xcolor}
\usepackage{cite}
\usepackage{subcaption}
\usepackage{enumitem}

\newcommand{\xvec}[1]{\mathbf{\boldsymbol{#1}}}
\newcommand{\xmat}[1]{\mathbf{\boldsymbol{#1}}}
\newcommand{\tvec}[1]{\underline{\xvec{#1}}}
\newcommand{\tmat}[1]{\underline{\xmat{#1}}}

\newcommand{\dx}{\Delta x}
\newcommand{\dy}{\Delta y}
\newcommand{\dt}{\Delta t}
\newcommand{\dtau}{\Delta \tau}

\newcommand{\iter}[2][i]{u_{#1}^{(#2)}}
\newcommand{\itvec}[1]{\xvec{u}^{(#1)}}

\newtheorem{theorem}{Theorem}

\newtheorem{lemma}[theorem]{Lemma}
\newtheorem{definition}{Definition}
\newtheorem{assumption}{Assumption}
\newtheorem{remark}{Remark}

\title{Conservative iterative methods for implicit discretizations of conservation laws}
\author{Philipp Birken$^{\mbox{\tiny\rm 1}}$, Viktor Linders$^{\mbox{\tiny\rm 1}}$}
\date{}
\begin{document}
\maketitle
\baselineskip=0.9
\normalbaselineskip
\vspace{-3pt}
\begin{center}{\footnotesize\em $^{\mbox{\tiny\rm 1}}$Centre for
    mathematical sciences, Lund University, Lund, Sweden.\\ email: philipp.birken\symbol{'100}na.lu.se \\ \qquad viktor.linders@math.lu.se }
\end{center}

\begin{abstract}
\noindent Conservation properties of iterative methods applied to implicit finite volume discretizations of nonlinear conservation laws are analyzed. It is shown that any consistent multistep or Runge-Kutta method is globally conservative. Further, it is shown that Newton's method, Krylov subspace methods and pseudo-time iterations are globally conservative while the Jacobi and Gauss-Seidel methods are not in general. If pseudo-time iterations using an explicit Runge-Kutta method are applied to a locally conservative discretization, then the resulting scheme is also locally conservative. However, the corresponding numerical flux can be inconsistent with the conservation law. We prove an extension of the Lax-Wendroff theorem, which reveals that numerical solutions based on these methods converge to weak solutions of a modified conservation law where the flux function is multiplied by a particular constant. This constant depends on the choice of Runge-Kutta method but is independent of both the conservation law and the discretization. Consistency is maintained by ensuring that this constant equals unity and a strategy for achieving this is presented. Experiments show that this strategy improves the convergence rate of the pseudo-time iterations.
\end{abstract}



\section{Introduction}

Conservation laws arise ubiquitously in the modelling of physical phenomena and their discretizations have been the subject of intense study; see e.g. \cite[Chapter 1]{leveque1992numerical} and \cite[Chapter 1]{leveque2002finite}. They derive their name from the fact that they describe the conservation of some quantities of interest over time. In computational fluid dynamics (CFD), these quantities are usually mass, momentum and energy. For convenience we always refer to the conserved quantity as the "mass" in this paper.

Many numerical schemes have been designed to mimic mass conservation. Throughout, we refer to such schemes as \emph{globally conservative}. As special cases, \emph{locally conservative} schemes dictate that local variations in the mass propagate from one computational cell to neighboring ones without "skipping" any cells (typically referred to as "conservative schemes" in the literature). In particular, finite volume methods are designed in part upon this principle, although many other schemes possess the same property \cite{fisher2013high}. A major result on the topic is the Lax-Wendroff theorem, see e.g. \cite{lax1959systems}, which provides sufficient conditions for numerical solutions of convergent and locally conservative discretizations to converge to weak solutions of the corresponding conservation law.

For stiff problems, e.g. the simulation of wall bounded viscous flows, \cite{miranker1973numerical}, implicit time stepping methods are used. Such discretizations result in systems of linear or nonlinear equations. In the context of fluid flow simulations, these systems are sparse and very large. Their solutions are approximated using iterative methods, see e.g. \cite{bassi2000gmres, bassi2011optimal, cockburn2005locally, birken2013preconditioning, birken2008preconditioner, blom2016comparison, birken2010nonlinear, birken2019preconditioned, birken2019subcell, birken2021numerical} and the references therein.

It is natural to ask whether such approximate solutions satisfy the conservation properties upon which the discretizations are based, and in particular, whether a Lax-Wendroff type result is available. It is typically claimed that this is not the case, e.g. in \cite{junqueira2014study}, where a version of a Newton-Jacobi iteration was noted to violate global conservation. In this paper we set out to answer this question in greater detail for some well-known (families of) iterative methods. We consider conservative implicit space-time discretizations whose solutions in each time step are approximated by a fixed number of iterations with different iterative methods.

After introducing relevant concepts and definitions in section \ref{sec:definitions}, we investigate the global conservation of iterative methods in section \ref{sec:mass_conservation}. It is shown that Newton's method, Krylov subspace methods and pseudo-time iterations using Runge-Kutta (RK) methods are globally conservative if the initial guess has correct mass. However, the Jacobi and Gauss-Seidel methods in general are not. In section \ref{sec:local_conservation} we focus on pseudo-time iterations using explicit RK methods. We show that such methods are locally conservative. By fixing the number of iterations and considering the limit of infinitesimal space-time increments, an extension of the Lax-Wendroff theorem is given. It turns out that the numerical solution in general converges to a weak solution of a modified conservation law for which the flux function is multiplied by a particular constant. An expression for this modification constant is given that only depends on the choice of Runge-Kutta method and the selected pseudo-time steps. A technique for ensuring that the constant equals one is presented, thereby ensuring consistency of the resulting scheme. Experiments indicate that the new technique improvess the convergence rate of the pseudo-time iterations. Numerical tests corroborate these findings and further suggest that the results hold for systems of conservation laws in multiple dimensions. Conclusions are drawn in section \ref{sec:conclusions}.

Throughout the paper we separate quantities that belong to a spatial discretization from those that do not. Any vector living on a grid with $m$ cells is denoted by a lower case bold letter, e.g. $\xvec{u} \in \mathbb{R}^m$. Similarly, any matrix operating on such a vector is represented by a bold upper case letter, e.g. $\xmat{A} \in \mathbb{R}^{m \times m}$. In Section \ref{sec:local_conservation}, we consider $s$-stage Runge-Kutta methods. Vector quantities spanning these stages are denoted by bold and underlined lower case letters, e.g. $\tvec{b} \in \mathbb{R}^s$. Matrices operating on the stages are represented by bold and underlined upper case letters e.g. $\tmat{A} \in \mathbb{R}^{s \times s}$.


\section{Conservation and conservative discretizations} \label{sec:definitions}

In this paper we consider numerical methods for conservation laws, meaning partial differential equations of the form
\begin{equation} \label{eq:conservation_law}
\begin{aligned}
u_t + \nabla \cdot f(u) &= 0, \quad t \in (t_0,t_e], \quad x\in\Omega \subset \mathbb{R}^d, \\
u(x,t_0) &= u_0(x),
\end{aligned}
\end{equation}
possibly with additional boundary conditions. Here, the $k$th component $u_k$ of the vector $u \in \mathbb{R}^q$ represents the concentration of a quantity, i.e. the total amount of that quantity in a domain is given by $\int_{\Omega} u_k \text{d}x$. The conservation law states that this amount is only changed by flow across the boundary of the domain. Assuming that the net flow across the boundary is zero, we thus have for all $t \in (t_0,t_e]$ that
\begin{equation}\label{eq:conservation}
  \int_{\Omega} u(x,t) \text{d}x = \int_{\Omega} u_0(x) \text{d}x.
\end{equation}

In this paper we focus on scalar conservation laws in 1D, i.e. the case when $d = q = 1$ in \eqref{eq:conservation_law}. Numerical experiments in section \ref{sec:local_conservation} suggest that the results of the paper can be generalized to systems in multiple dimensions, however this task is deferred for future work. Further, we restrict our analysis to the Cauchy problem for \eqref{eq:conservation_law}, or the problem with periodic boundary conditions. Thus, the conservation law of interest becomes
\begin{equation} \label{eq:conservation_law_1D}
\begin{aligned}
u_t + f_x &= 0, \quad t \in (t_0,t_e], \quad x \in \Omega, \\
u(x,t_0) &= u_0(x),
\end{aligned}
\end{equation}
where either $\Omega = (-\infty, \infty)$ or $\Omega = (a,b]$ and $u(a) = u(b)$. Which case is considered will be clear from context.

A very successful line of research in computational fluid dynamics is to construct numerical methods that respect \eqref{eq:conservation} on a discrete level. Herein, we predominantly consider finite volume methods that utilize the implicit Euler method as temporal discretization. Let the computational grid be given by $(x_i, t_n) = (i \Delta x, n \Delta t)$ with $n \geq 0$. This grid may be either infinite or finite and periodic in space, i.e. $x_{i+m} = x_i$ for some positive integer $m$. We consider discretizations on the form
\begin{equation} \label{eq:FV}
    \frac{u_i^{n+1} - u_i^n}{\dt} + \frac{1}{\dx} \left( \hat{f}_{i+\frac{1}{2}} - \hat{f}_{i-\frac{1}{2}} \right) = 0.
\end{equation}
Here, $u_i^n$ approximates the solution $u(x,t_n)$ of \eqref{eq:conservation_law_1D} in the computational cell $\Omega_i$. The notation $\hat{f}_{i+\frac{1}{2}}(\xvec{u})$ is used as a placeholder for a generic numerical flux of the form $\hat{f}(u_{i-p}^{n+1},\dots,u_{i+q}^{n+1})$, where $p,q \in \mathbb{N}$ determine the bandwidth of the stencil. Multiplying \eqref{eq:FV} by $\dx$ and summing over all $i$ reveals that
\begin{equation} \label{eq:FV_global_conservation}
\dx \sum_i u_i^{n+1} = \dx \sum_i u_i^n.
\end{equation}
Comparing this with \eqref{eq:conservation} shows that the finite volume scheme discretely mimics conservation of mass.

\subsection{Local conservation}

It is the telescoping nature of the spatial terms in \eqref{eq:FV} that leads to the preceeding result. In the finite volume community, this property is typically referred to simply as conservation. However, to distinguish it from other mass conservative schemes, we will henceforth refer to this property as \emph{local conservation}:

\begin{definition}
A discretization of \eqref{eq:conservation_law} that can be written in the form \eqref{eq:FV} is said to be locally conservative.
\end{definition}

Local conservation plays a central role in the Lax-Wendroff theorem, which considers the Cauchy problem for \eqref{eq:conservation_law_1D}. First, the notion of \emph{consistency} must be specified:

\begin{definition} \label{def:consistency}
A numerical flux $\hat{f}$ is consistent with $f$ if it is Lipschitz continuous in each argument and if $\hat{f}(u,\dots,u) = f(u)$.
\end{definition}

A numerical flux function is \emph{inconsistent} if it fails to satisfy Definition \ref{def:consistency}. We will see examples of this in Section \ref{sec:local_conservation}.

The Lax-Wendroff theorem applies to the Cauchy problem for \eqref{eq:conservation_law_1D} and considers locally conservative discretizations with consistent numerical flux. If the numerical solution of such a scheme converges to a function $u$ in the limit of vanishing $\dx$ and $\dt$, the theorem provides sufficient conditions for $u$ to be a weak solution of the conservation law \eqref{eq:conservation_law} \cite[Chapter 12]{leveque1992numerical}. More precisely, consider a sequence of grids $(\Delta x_\ell, \Delta t_\ell)$ such that $\Delta x_\ell, \Delta t_\ell \rightarrow 0$ as $\ell \rightarrow \infty$. Let $\mathcal{U}_\ell(x,t)$ denote the piecewise constant function that takes the solution value $u_i^n$ in $(x_i, x_{i+1}] \times (t_{n-1}, t_n]$ on the $\ell$th grid. We make the following assumptions:

\begin{assumption} \label{assumptions} 
\hfill
\begin{enumerate}
    \item There is a function $u(x,t)$ such that over every bounded set $\Omega = [a,b] \times [0,T]$ in $x$-$t$ space,
    $$
    \| \mathcal{U}_\ell(x,t) - u(x,t) \|_{1,\Omega} \rightarrow 0 \quad \text{as} \quad \ell \rightarrow \infty.
    $$
    \item For each $T$ there is a constant $R > 0$ such that
    $$
    TV(\mathcal{U}_\ell(\cdot,t)) < R \quad \text{for all} \quad 0 \leq t \leq T, \quad \ell = 1, 2, \dots
    $$
\end{enumerate}
\end{assumption}

The first of these assumptions asserts that the numerical solution is convergent in the $L^1$-norm with limit $u$. The second one states that the total variation of the numerical solution remains bounded independently of the grid. The Lax-Wendroff theorem may now be stated as follows:

\begin{theorem}[Lax-Wendroff] \label{thm:Lax_Wendroff}
Consider the locally conservative discretization \eqref{eq:FV}, suppose that the numerical flux $\hat{f}$ is consistent and that Assumption \ref{assumptions} is satisfied. Then, $u(x,t)$ is a weak solution of \eqref{eq:conservation_law}.
\end{theorem}

The theorem and its assumptions require that the linear or nonlinear systems arising in \eqref{eq:FV} are solved exactly. However, in practice these systems will be solved approximately using iterative methods. The natural question that now arises is: Do iterative solvers maintain the convergence to weak solutions?

\subsection{Global conservation}

To answer this question we must establish if iterative methods preserve local conservation. However, local conservation is a special case of the more general notion of \emph{global conservation}. Let us therefore consider a finite and periodic grid. Discretizing \eqref{eq:conservation_law_1D} in space on a mesh with cells $\Omega_i$ of volume $|\Omega_i|$, we arrive at an initial value problem
\begin{equation} \label{eq:semiDiscretization}
\begin{aligned}
\xvec{u}_t &= \hat{\xvec{f}}(\xvec{u}), \quad t \in (t_0,t_e], \\
\xvec{u}(t_0) &= \xvec{u}_0,
\end{aligned}
\end{equation}
with $\xvec{u}(t), \hat{\xvec{f}} \in \mathbb{R}^m$.

\begin{definition} \label{def:conservative}
A discretization $\hat{\xvec{f}} = (\hat{f}_1(\xvec{u}), \dots, \hat{f}_m(\xvec{u}))^\top$ is globally conservative if
$$
\sum_{i=1}^m |\Omega_i| \hat{f}_i(\xvec{u}) = 0.
$$
\end{definition}

Global conservation implies that \eqref{eq:conservation} is fulfilled on a semidiscrete level by \eqref{eq:semiDiscretization}:
$$
\sum_{i=1}^m|\Omega_i|u_i(t)=\sum_{i=1}^m|\Omega_i|u_i(t_0), \quad \forall t \in [t_0,t_e].
$$

Suppose next that a time integration method is used to solve \eqref{eq:semiDiscretization}.

\begin{definition} \label{def:conservativeTimeIntegration}
A time integration method applied to the conservative semidiscretization \eqref{eq:semiDiscretization} is globally conservative if
\begin{equation*}
\sum_{i=1}^m|\Omega_i|u_i^{n+1}=\sum_{i=1}^m|\Omega_i|u_i^n
\end{equation*}
holds for each time step $n=0,1,\dots$
\end{definition}

Note that the conservation property \eqref{eq:FV_global_conservation} derived for the locally conservative finite volume discretization \eqref{eq:FV} is a special case of Definition \ref{def:conservativeTimeIntegration}, where $\Omega_i = \dx$ for every $i$. Global conservation is thus necessary for local conservation, which in turn is necessary for the Lax-Wendroff Theorem.

In the CFD community, one hears from time to time the claim that implicit methods are not conservative. However, this is not true: All consistent linear multistep methods are globally conservative. To see this, consider a generic $s$-step linear multistep method
\begin{equation} \label{eq:multi_step}
\sum_{j=0}^s a_j \xvec{u}^{n+j} = \dt \sum_{j=0}^s b_j \hat{\xvec{f}}(\xvec{u}^{n+j}),
\end{equation}
where $a_j$ and $b_j$ are given method dependent coefficients. Suppose that $\sum_{i=1}^m |\Omega_i| u_i^{n+j} = \sum_{i=1}^m |\Omega_i| u_i^0$  for $j = 0, \dots, s-1$. Multiplying \eqref{eq:multi_step} by $|\Omega_i|$ and summing over all cells gives
\begin{align*}
\sum_{i=1}^m |\Omega_i| a_s u_i^{n+s} &= -\sum_{j=0}^{s-1} a_j \sum_{i=1}^m |\Omega_i| u_i^{n+j} + \dt \sum_{j=0}^s b_j \underbrace{\sum_{i=1}^m |\Omega_i| \hat{f}_i(\xvec{u}^{n+j})}_{=0} \\
&= -\sum_{i=1}^m |\Omega_i| u_i^0 \sum_{j=0}^{s-1} a_j.
\end{align*}
Consistent multistep methods satisfy $a_s = 1$ and $a_0 + \dots + a_{s-1} = -1$. Hence, global conservation is ensured.

Similarly, all Runge-Kutta (RK) methods are globally conservative. Consider an $s$-stage RK method
\begin{equation} \label{eq:RK_method}
\begin{aligned}
\xvec{u}^{n+1} &= \xvec{u}^n + \dt \sum_{j=1}^s b_j \xvec{k}_j, \\
\xvec{k}_j &= \hat{\xvec{f}}\left( \xvec{u}^n + \dt \sum_{l=1}^s a_{j,l} \xvec{k}_l \right), \quad j = 1, \dots, s,
\end{aligned}
\end{equation}
where $a_{j,l}$ and $b_j$ are given by the method. Since $\hat{\xvec{f}}$ is globally conservative it follows from Definition \ref{def:conservative} that the mass of $\xvec{k}_j$ is zero and therefore that
$$
\sum_{i=1}^m |\Omega_i| u_i^{n+1} = \sum_{i=1}^m |\Omega_i| u_i^n + \dt \sum_{j=1}^s b_j \sum_{i=1}^m |\Omega_i| k_{j_i} = \sum_{i=1}^m |\Omega_i| u_i^n.
$$
Hence, RK methods are globally conservative.

Exponential integrators have been proven to be conservative as well \cite{birken2016stability}, essentially due to the following useful lemma:

\begin{lemma}\label{lemma:conservativeJacobian}
For any $\xvec{y} \in \mathbb{R}^m$ the Jacobian $\hat{\xvec{f}}'(\xvec{u}) = \frac{\partial \hat{\xvec{f}}}{\partial \xvec{u}} \in \mathbb{R}^{m \times m}$ of the globally conservative discretization $\hat{\xvec{f}}(\xvec{u})$ satisfies
\begin{equation}
 \sum_{i=1}^m|\Omega_i| \left( \hat{\xvec{f}}' \xvec{y} \right)_i=0.
\end{equation}
\end{lemma}

\begin{proof}
See \cite{straub2017adopting}.
\end{proof}

In the remainder we also need the following result:

\begin{lemma}\label{lemma:conservativelinearsystem}
Let $\hat{\xmat{f}}'(\xvec{u}) \in \mathbb{R}^{m \times m}$ be the Jacobian of a conservative discretization $\hat{\xvec{f}}(\xvec{u})$ and let $\alpha$ be any scalar. The solution $\xvec{x}$ of a linear system of the form
\begin{equation} \label{eq:genericLinearSystem}
\left( \xmat{I} + \alpha \hat{\xmat{f}}' \right) \xvec{x} = \xvec{b},
\end{equation}
satisfies
$$
\sum_{i=1}^m |\Omega_i| x_i = \sum_{i=1}^m |\Omega_i| b_i.
$$
\end{lemma}

\begin{proof}
Multiplying the $i$th element of \eqref{eq:genericLinearSystem} by $|\Omega_i|$ and summing over all $i$ gives
$$
\sum_{i=1}^m |\Omega_i| b_i = \sum_{i=1}^m |\Omega_i| \left( \left( \xmat{I} + \alpha \hat{\xmat{f}}' \right) \xvec{x} \right)_i = \sum_{i=1}^m |\Omega_i| x_i + \alpha \sum_{i=1}^m |\Omega_i| \left( \hat{\xmat{f}}' \xvec{x} \right)_i.
$$
By Lemma \ref{lemma:conservativeJacobian} the final term on the right-hand side is zero.
\end{proof}


\section{Globally conservative iterative solvers} \label{sec:mass_conservation}

We begin by asking: Do iterative solvers respect global conservation? Since global conservation is necessary for local conservation, which in turn is used in Theorem \ref{thm:Lax_Wendroff}, a negative answer will immediately exclude the possibility to extend the Lax-Wendroff theorem to incorporate a particular iterative method.

Discretizing \eqref{eq:semiDiscretization} in time using the implicit Euler method results in a nonlinear equation system of the form
\begin{equation} \label{eq:nonlineq}
\xvec{u} - \alpha \hat{\xvec{f}}(\xvec{u}) = \xvec{u}^n.
\end{equation}
We concern ourselves with a set of iterates $\itvec{k}$ with limit $\xvec{u}^{n+1}$ as $k \rightarrow \infty$.

\begin{definition} \label{def:globally_conservative_iterative_method}
An iterative method with iterates $\itvec{k}$ is said to be globally conservative if 
$$
\sum_{i=1}^m |\Omega_i| \iter{k+1} = \sum_{i=1}^m |\Omega_i| u_i^n, \qquad k = 0, 1, \dots
$$
\end{definition}

In the following subsections we investigate the conservation properties of some of the most common iterative methods.

\subsection{Newton's Method}

Newton's method finds an approximate solution to \eqref{eq:nonlineq} by solving the sequence of linear systems
\begin{equation} \label{eq:NewtonIteration}
\begin{aligned}
\left( \xmat{I} - \alpha \hat{\xmat{f}}' \left( \itvec{k} \right) \right) \Delta \xvec{u} &= \xvec{u}^n - \itvec{k} + \alpha \hat{\xmat{f}}\left( \itvec{k} \right), \\
\itvec{k+1} &= \itvec{k} + \Delta \xvec{u},
\end{aligned}
\end{equation}
for  $k = 0, 1,\dots$ Since $\hat{\xmat{f}}$ is conservative it has zero mass by Definition \ref{def:conservative}. If $\itvec{k}$ and $\xvec{u}^n$ have the same mass, the entire right-hand side of the first equation in \eqref{eq:NewtonIteration} has zero mass. Hence, by Lemma \ref{lemma:conservativelinearsystem}, the increment $\Delta \xvec{u}$ also has zero mass. Thus, the mass of $\itvec{k+1}$ is the same as that of $\itvec{k}$, which by assumption is the same as that of $\xvec{u}^n$. Hence, Newton's method is globally conservative.

In large scale practical applications, each linear system in \eqref{eq:NewtonIteration} is solved approximately using other iterative methods. In the following, we consider a generic linear system
\begin{equation}\label{eq:lineq}
(\xmat{I} - \alpha \xmat{A}) \xvec{u} = \xvec{b},
\end{equation}
that may either represent a Newton iteration or a linear discretization of a conservation law. In line with Lemma \ref{lemma:conservativeJacobian}, it is assumed that the mass of $\xmat{A} \xvec{y}$ is zero for any vector $\xvec{y} \in \mathbb{R}^m$.

\subsection{Stationary linear methods}

Any iterative method for solving \eqref{eq:lineq} that can be written in the form
\begin{equation} \label{eq:stationaryLinearMethod}
\itvec{k+1} = \xmat{M} \itvec{k} + \xmat{N}^{-1} \xvec{u}^n
\end{equation}
for some matrices $\xmat{M}$ and $\xmat{N}$, is termed a stationary linear method. Well known examples include the Richardson, Jacobi and Gauss-Seidel iterations.

\subsubsection{Richardson iteration}

For the Richardson iteration, $\xmat{M} = \xmat{I} - \theta (\xmat{I} - \alpha \xmat{A})$ and $\xmat{N}^{-1} = \theta \xmat{I}$, where $\theta \in \mathbb{R}$ is a fixed parameter. With some minor rearrangements, the iteration can thus be expressed as
$$
\itvec{k+1} = \itvec{k} + \theta (\xvec{b} - (\xmat{I} - \alpha \xmat{A}) \itvec{k}).
$$
Multiplying the $i$th element by $|\Omega_i|$ and summing over $i$ results in
$$
\sum_{i=1}^m|\Omega_i| \iter{k+1} = \sum_{i=1}^m |\Omega_i| \iter{k} + \theta \sum_{i=1}^m |\Omega_i| (b_i - \iter{k}) + \alpha \theta \sum_{i=1}^m |\Omega_i| (\xmat{A} \itvec{k})_i.
$$
By assumption, the matrix-vector product $\xmat{A} \itvec{k}$ has zero mass, hence the final term vanishes. Further, if $\itvec{k}$ has the same mass as $\xvec{b}$, then the second term on the right-hand side vanishes. By induction it follows that the Richardson iteration is globally conservative.

\subsubsection{Jacobi iteration}

Consider again the linear system \eqref{eq:lineq}. With the matrix decomposition $\xmat{I}-\alpha \xmat{A} = \xmat{I} - \alpha (\xmat{D} + \xmat{L} + \xmat{U})$, where $\xmat{D}$, $\xmat{L}$ and $\xmat{U}$ are the diagonal, lower and upper triangular parts of $\xmat{A}$ respectively, the Jacobi iteration is obtained by setting $\xmat{M} = \alpha (\xmat{I} - \alpha \xmat{D})^{-1} (\xmat{L} + \xmat{U})$ and $\xmat{N}^{-1} = (\xmat{I} - \alpha \xmat{D})^{-1}$. Inserting this into \eqref{eq:stationaryLinearMethod}, multiplying by $\xmat{I} - \alpha \xmat{D}$, then adding and subtracting $\alpha \xmat{D} \itvec{k}$, gives after some rearrangement
$$
\itvec{k+1} = \xvec{b} + \alpha \xmat{D} ( \itvec{k+1} - \itvec{k}) + \alpha \xmat{A} \itvec{k}.
$$
Multiplying the $i$th element by $|\Omega_i|$ and summing over $i$ results in
\begin{equation} \label{eq:Jacobi_mass}
\sum_{i=1}^m |\Omega_i| \iter{k+1} = \sum_{i=1}^m |\Omega_i| b_i + \alpha \sum_{i=1}^m |\Omega_i| a_{i,i} (\iter{k+1} - \iter{k}) + \alpha \sum_{i=1}^m |\Omega_i| (\xmat{A} \itvec{k})_i.
\end{equation}
The final term vanishes by the assumption. However, the second term on the right-hand side is in general non-zero and thus introduces a conservation error
$$
\alpha \sum_{i=1}^m |\Omega_i| a_{i,i} (\iter{k+1} - \iter{k}).
$$
The Jacobi method is therefore not globally conservative in general. 

However, consider the special case when the diagonal $\xmat{D}$ of $\xmat{A}$ is a scalar multiple of the identity matrix, say $\xmat{D} = a \xmat{I}$. Then $a_{i,i} = a$ for $i = 1, \dots, m$. If $\itvec{k}$ has the same mass as $\xvec{b}$ we find, after rearrangement of \eqref{eq:Jacobi_mass}, that
$$
(1 - \alpha a) \sum_{i=1}^m |\Omega_i| \iter{k+1} = (1 - \alpha a) \sum_{i=1}^m |\Omega_i| b_i.
$$
Thus, $\itvec{k+1}$ has correct mass, at least as long as $\alpha a \neq 1$. Hence, the Jacobi iteration is conservative in this special case.

\subsubsection{Gauss-Seidel iteration}

The Gauss-Seidel iteration is defined by $\xmat{M} = \alpha (\xmat{I} - \alpha (\xmat{D} + \xmat{L}))^{-1} \xmat{U}$ and $\xmat{N}^{-1} = (\xmat{I} - \alpha (\xmat{D} + \xmat{L}))^{-1}$. Inserting this into \eqref{eq:stationaryLinearMethod}, multiplying by $\xmat{I} - \alpha (\xmat{D} + \xmat{L})$, then adding and subtracting $\alpha \xmat{U} \itvec{k+1}$ gives
\begin{equation*}
\itvec{k+1} = \xvec{b} - \alpha \xmat{U} (\itvec{k+1} - \itvec{k}) + \alpha A \itvec{k+1}.
\end{equation*}
Following the same procedure as previously, we multiply the $i$th element by $|\Omega_i|$ and sum over $i$ to obtain
\begin{align*}
\sum_{i=1}^m |\Omega_i| \iter{k+1} &= \sum_{i=1}^m |\Omega_i| b_i + \alpha \sum_{i=1}^m |\Omega_i| \sum_{j=i+1}^m a_{i,j} (\iter[j]{k+1} - \iter[j]{k}) \\
&+ \alpha \sum_{i=1}^m |\Omega_i| (\xmat{A} \itvec{k+1})_i.
\end{align*}
By assumption, the final term vanishes. However, the second term on the right-hand side is in general non-zero. Thus, the Gauss-Seidel method is not globally conservative and the conservation error is given by
$$
\alpha \sum_{i=1}^m |\Omega_i| \sum_{j=i+1}^m a_{i,j} (\iter[j]{k+1} - \iter[j]{k}).
$$
However, note that for problems where $\xmat{A}$ is lower triangular the Gauss-Seidel iteration is indeed conservative. In fact, for such problems the method is by definition exact.

\subsection{Krylov subspace methods}

Given an initial guess $\itvec{0}$ for the solution of the linear system \eqref{eq:lineq}, the $k$th iteration of a Krylov subspace method belongs to the space
\begin{align*}
\mathcal{K}_k &= \itvec{0} + \text{span} \left\{ \xvec{r}^{(0)}, (\xmat{I} - \alpha \xmat{A}) \xvec{r}^{(0)}, \dots, (\xmat{I} - \alpha \xmat{A})^{k-1} \xvec{r}^{(0)} \right\} \\
&= \itvec{0} + \text{span} \left\{ \xvec{r}^{(0)}, \xmat{A} \xvec{r}^{(0)}, \dots, \xmat{A}^{k-1} \xvec{r}^{(0)} \right\},
\end{align*}
where $\xvec{r}^{(0)} = \xvec{b} - (\xmat{I} - \alpha \xvec{A}) \itvec{0}$ is the initial residual. Examples of Krylov subspace methods are GMRES, CG and BICGSTAB \cite[Chapters 6--7]{saad2003iterative}. 

We may express the $k$th iteration $\itvec{k}$ as
$$
\itvec{k} = \itvec{0} + \sum_{l=0}^{k-1} c_l \xmat{A}^l \xvec{r}^{(0)},
$$
for some coefficients $c_l, \, l=0,\dots,k-1$ that depend on, and are chosen by the method. It follows that
$$
\sum_{i=1}^m |\Omega_i| \iter{k} = \sum_{i=1}^m |\Omega_i| \iter{0} + c_0 \sum_{i=1}^m |\Omega_i| r_i^{(0)} + \sum_{l=1}^{k-1} c_l \sum_{i=1}^m |\Omega_i| (\xmat{A}^l \xvec{r}^{(0)})_i.
$$
The final term vanishes by assumption. Note that the mass of $\xvec{r}^{(0)}$ is zero if the mass of $\itvec{0}$ equals that of $\xvec{b}$. Thus, the second term on the right-hand side also vanishes. All Krylov subspace methods are therefore conservative if the initial guess is chosen to have the same mass as the right-hand side.

Typically a preconditioner $\xmat{P}^{-1}$ will pre- or post-multiply $\xvec{A}$ in Krylov subspace methods. We will not delve into this subject here but merely observe that if $\xvec{y}$ and $\xmat{P}^{-1} \xvec{y}$ have the same mass for any choice of $\xvec{y}$, then the above analysis applies without modification.

\subsection{Multigrid methods}

Multigrid methods combine two methods, namely a smoother and a coarse grid correction (CGC). The idea is to separate the residual into high and low frequency components. As a smoother, an iterative method is applied, designed to effectively damp the high frequency components.

The role of the CGC is to remove the low frequency components of the residual. The residual is mapped to a coarser grid using a restriction operator, $\xmat{R}$. Smoothing is then applied and a prolongation operator $\xmat{P}$ is used to reconstruct the residual on the fine grid. Finally, the fine grid solution is updated by adding the correction to the previous iterate. The procedure can be applied to a hirearchy of grids, thereby resulting in a multigrid method. If we consider only two grid levels and assume that the system is solved exactly on the coarser one, then the coarse grid correction for the linear problem \eqref{eq:lineq} has the form
\begin{equation} \label{eq:coarse_grid_correction}
\itvec{k+1} = (\xmat{I} - \xmat{P} (\xmat{I} - \alpha \xmat{A})_{\ell-1}^{-1} \xmat{R} (\xmat{I} - \alpha \xmat{A})_\ell) \itvec{k} + \xmat{P} (\xmat{I} - \alpha \xmat{A})_{\ell-1}^{-1} \xmat{R} \xvec{b}.
\end{equation}
Here, the notation $(\cdot)_\ell$ and $(\cdot)_{\ell-1}$ denote matrices operating on the fine and the coarse grid respectively.

The easiest way to make a multigrid method globally conservative is to choose its components to be globally conservative. The smoother can be any of the globally conservative iterations discussed so far. By Lemma \ref{lemma:conservativelinearsystem}, the inverted matrices in \eqref{eq:coarse_grid_correction} are mass conserving. Thus, it remains to choose globally conservative restriction and prolongation operators $\xmat{R}$ and $\xmat{P}$. This is achieved using agglomeration. The restriction is performed by agglomerating a number of neighboring cells by constructing a volume weighted average of fine grid values. Conversely, the prolongation is done by injecting coarse grid values at multiple fine grid points. In fact, this is the standard choice of multigrid method in CFD since it results in faster convergence than other alternatives. For details, see \cite{birken2012habilitation}.

\subsection{Pseudo-time iterations}

Pseudo-time iterations are obtained by adding a pseudo-time derivative term to the algebraic equations \eqref{eq:nonlineq} (or \eqref{eq:lineq}),
\begin{equation} \label{eq:pseudo_time_IVP}
\frac{\partial \xvec{u}}{\partial \tau} + \xvec{u} - \alpha \hat{\xmat{f}}(\xvec{u}) = \xvec{u}^n, \qquad \xvec{u}(0) = \xvec{u}_0.
\end{equation}
Here, $\xvec{u}_0$ is initial data that must be provided, e.g. $\xvec{u}_0 = \xvec{u}^n$. The idea is that the system eventually should reach a steady state as $\tau \rightarrow \infty$ where the pseudo-time derivative vanishes \cite{jameson1981numerical}, resuting in a solution to the original nonlinear system \eqref{eq:nonlineq}. Any time integration method can in principle be applied to the initial value problem \eqref{eq:pseudo_time_IVP}. Explicit RK methods are globally conservative by construction. Implicit methods result in new systems of equations whose solutions are once again approximated using iterative methods \cite{birken2019preconditioned}. Whether the resulting approximation is conservative depends on the choice of method. Any one of the conservative methods discussed so far can be applied in principle. We will return to pseudo-time iterations in Section \ref{sec:local_conservation}, where more details are provided.


\subsection{Numerical validation}

To validate the results of the preceding section we consider two simple experiments. The first is the linear advection equation with periodic boundary conditions,
\begin{equation} \label{eq:advection_equation}
\begin{aligned}
u_t + u_x &= 0, \quad x \in (-1.5, 1.5], \quad t \in (0,6], \\
u(x,0) &= \exp{(-50x^2)}, \\
u(-1.5,t) &= u(1.5,t).
\end{aligned}
\end{equation}
We discretize using the finite volume scheme \eqref{eq:FV} with a central numerical flux $\hat{f}_{i+\frac{1}{2}} = (u_{i+1}^{n+1} - u_i^{n+1})/2$. This results in a linear system of the form \eqref{eq:lineq} with $\alpha = -\dt / \dx$ and $\xmat{A} = \text{Tridiag} \left( -\frac{1}{2}, 0, \frac{1}{2} \right)$. Here, $\text{Tridiag}(\cdot)$ refers to a tridiagonal matrix with periodic wrap-around. Note that the diagonal of $\xmat{A}$ is constant. We therefore expect the Jacobi method to be conservative for this particular discretization.

The second experiment is Burgers' equation with periodic boundary conditions,
\begin{equation} \label{eq:Burgers_equation}
\begin{aligned}
u_t + \left( \frac{u^2}{2} \right)_x &= 0, \quad x \in (-1.5, 1.5], \quad t \in (0,0.6], \\
u(x,0) &= \exp{(-x^2)}, \\
u(-1.5,t) &= u(1.5,t).
\end{aligned}
\end{equation}
Again, we use the finite volume scheme \eqref{eq:FV}, however with the upwind flux $\hat{f}_{i+\frac{1}{2}} = \left( u_i^{n+1} \right)^2/2$. The result is a nonlinear system in the form \eqref{eq:nonlineq} with $\alpha = -\dt$ and $\hat{f}_i(\xvec{u}^{n+1}) = \left( (u_i^{n+1})^2 - (u_{i-1}^{n+1})^2 \right)/2\dx$. The Jacobian of this discretization is lower triangular except for a single element in the top right corner. Thus, we expect the Gauss-Seidel method to be nearly conservative if the grid is sufficiently fine.

Burgers' equation \eqref{eq:Burgers_equation} is solved using Newton's method with initial guess $\xvec{u}^n$. The resulting linear systems are either solved exactly or using the following iterative methods with $\xvec{0}$ as initial guess:
\begin{description}[leftmargin=!,labelwidth=1cm]
\item[(R)] The Richardson iteration using $\theta = 0.5$.
\item[(J)] The Jacobi method.
\item[(GS)] The Gauss-Seidel method.
\item[(GM)] GMRES without restarts or preconditioning.
\item[(CGC)] A two-level coarse grid correction.
\item[(H)] Pseudo-time iterations using Heun's method with $\dtau = 0.5$.
\end{description}
As prolongation and restriction operators for the CGC, agglomeration gives
$$
\xmat{R} = \frac{1}{2}
\begin{bmatrix}
1 & 1 & & & & \\
& & 1 & 1 & & \\
& & & \ddots & &
\end{bmatrix},
\quad \xmat{P} = 2 \xmat{R}^\top.
$$
The same methods are used directly on the advection problem \eqref{eq:advection_equation} with $\xvec{u}^n$ as initial guess.

First we consider very coarse discretizations with $\dx = \dt = 0.5$ for both problems. A single iteration is used with each method and the mass error and residual is computed after one time step. The results are shown in Table \ref{tab:mass_error}. Here, any error smaller than $10^{-15}$ is denoted as zero. The Richardson iteration, GMRES, CGC and the pseudo-time iterations are conservative for both problems. Similarly, the exact Newton method (Exact) is conservative. As expected, the Jacobi method is conservative for the  advection problem but not for Burgers' equation.  Gauss-Seidel is not conservative for any of the two problems.

\begin{table}[h!]
\begin{center}
\caption{Mass errors and residuals after one time step with $\dt = \dx = 0.5$ for the advection and Burgers' equations using a single iteration. Computed values smaller than $10^{-15}$ are denoted by $0$.}
\begin{tabular}{ |l|c|c|c|c|c|c|c| } 
 \hline 
{\bf Advection equation} & (Exact) & (R) & (J) & (GS) & (GM) & (CGC) & (H) \\
 \hline
 Mass error & & $0$ & $0$ & $-0.094$ & $0$ & $0$ & $0$ \\
 \hline
 Residual & & $0.331$ & $0.433$ & $0.256$ & $0.327$ & $0.162$ & $0.287$ \\
 \hline
 \multicolumn{4}{c}{\vspace{0.1cm}}\\
 \hline
 {\bf Burgers' equation} & (Exact) & (R) & (J) & (GS) & (GM) & (CGC) & (H) \\
 \hline
 Mass error & $0$ & $0$ & $0.031$ & $-0.034$ & $0$ & $0$ & $0$ \\
 \hline
 Residual & $0.897$ & $0.895$ & $0.901$ & $0.862$ & $0.898$ & $0.924$ & $0.878$ \\
 \hline
\end{tabular}
\label{tab:mass_error}
\end{center}
\end{table}

Next, we repeat the experiment with well resolved discretizations. For the advection problem we choose $\dt = \dx = 0.006$ and for Burgers' equation $\dt = \dx = 0.003$. In the latter case, two Newton iterations are used per time step. For each  of the other methods, 5 iterations are used except for CGC where a single iteration is performed.

\begin{figure}[t!]
    \centering
    \begin{subfigure}{0.49\textwidth}
    \includegraphics[width=\textwidth]{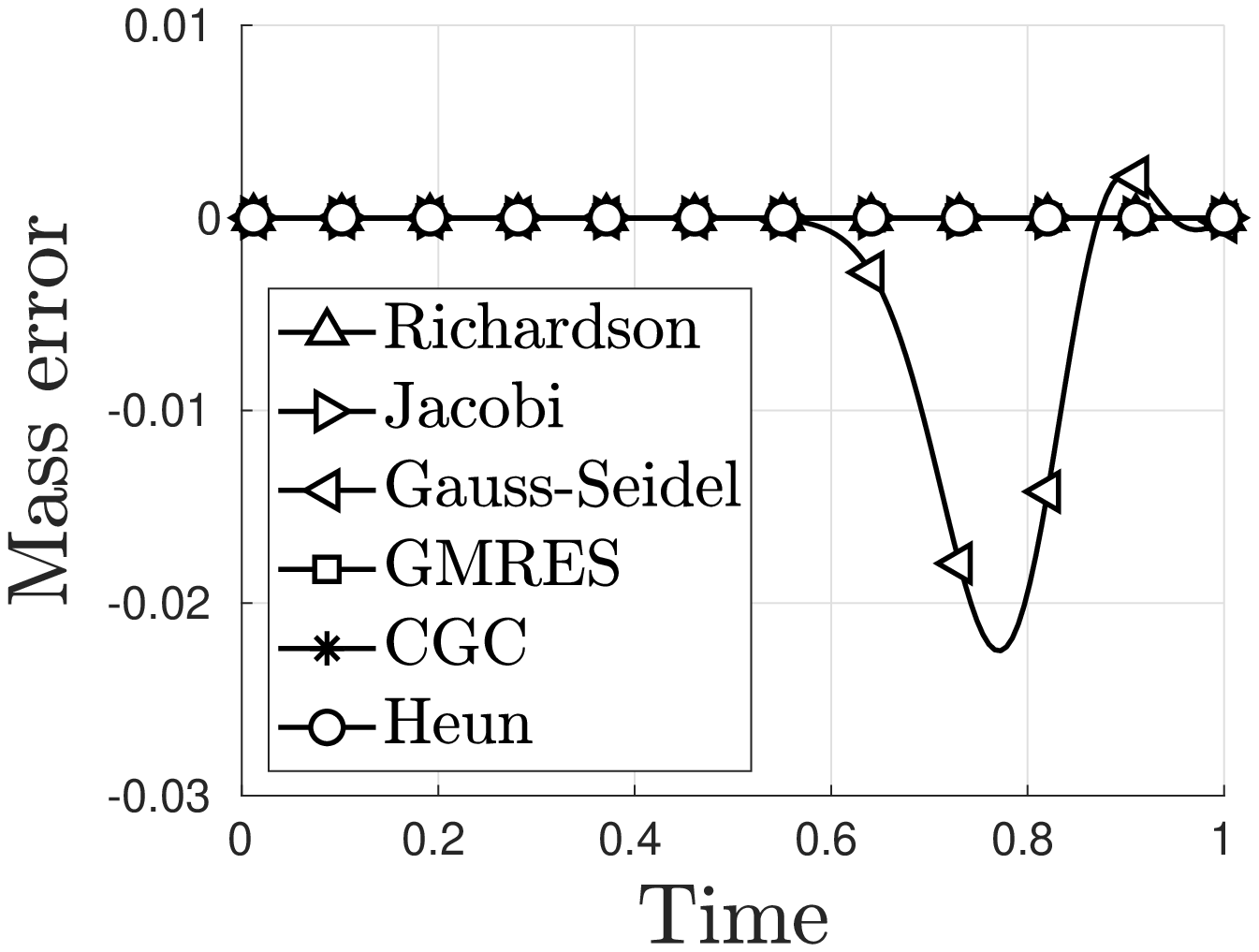}
    \caption{Linear advection.}
    \label{fig:mass_conservation_advection}
    \end{subfigure}
    \begin{subfigure}{0.49\textwidth}
    \includegraphics[width=\textwidth]{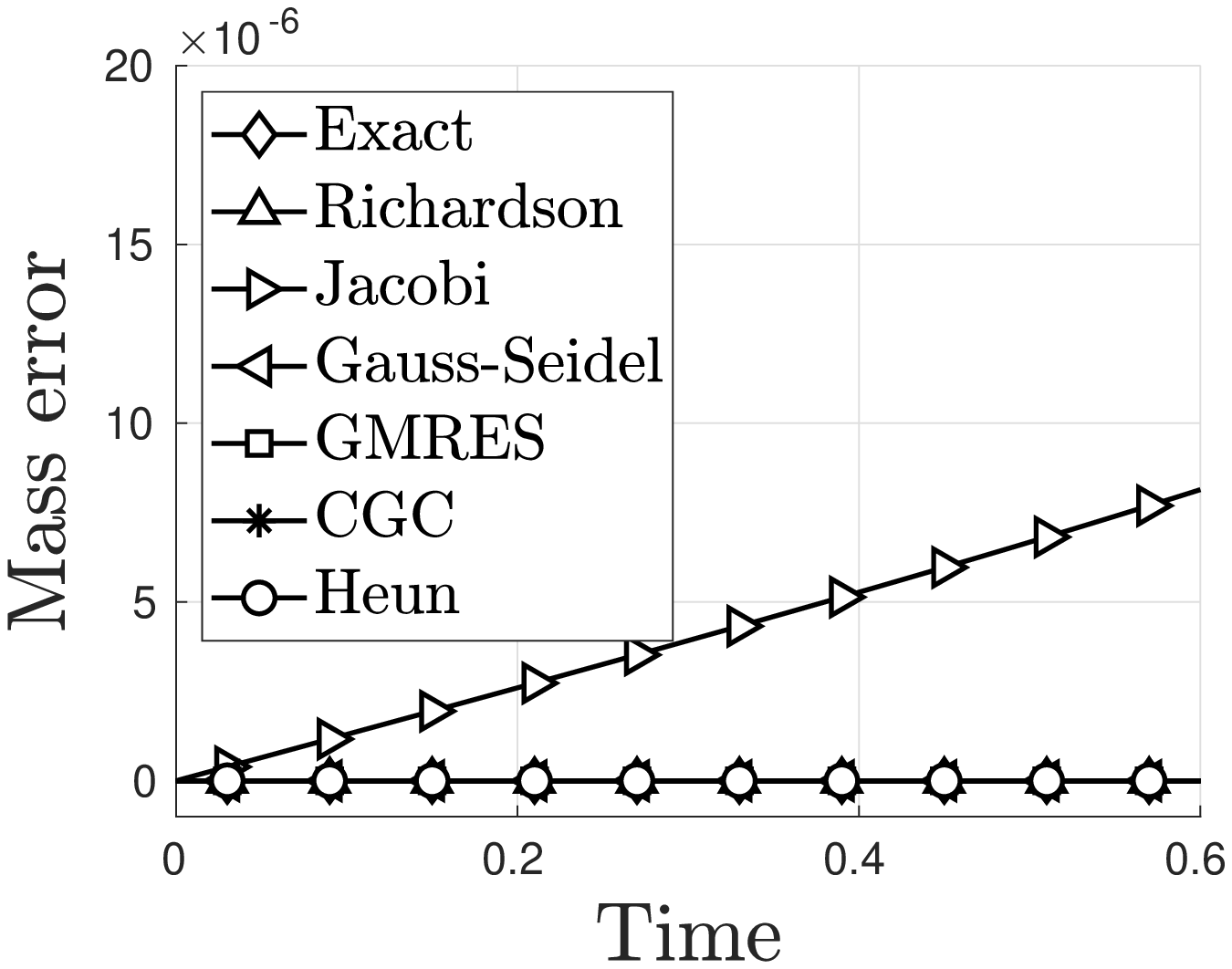}
    \caption{Burgers' equation.}
    \label{fig:mass_conservation_Burgers}
    \end{subfigure}
    \caption{Mass error for iterative methods applied to the advection equation \eqref{eq:advection_equation} and Burgers' equation \eqref{eq:Burgers_equation}.}
    \label{fig:mass_conservation}
\end{figure}

The total mass is computed in each time step and compared with the mass of the initial data. Fig. \ref{fig:mass_conservation} shows the results. The main difference to the previous test is that Gauss-Seidel has a non-detectible mass error for Burger's equation. This is in line with our expectations since the Jacobian only has a single nonzero element above the main diagonal.


\section{Local conservation of pseudo-time iterations} \label{sec:local_conservation}

Having covered global conservation in the previous section, we now switch focus to local conservation. The Jacobi and Gauss-Seidel methods cannot be locally conservative since they are not globally conservative in general. For the other methods considered so far, the question remains open.

In what follows, we restrict our attention to the Cauchy problem for \eqref{eq:conservation_law_1D}. We focus on pseudo-time iterations, show that these methods are locally conservative and prove an extension of the Lax-Wendroff theorem. Noting that the Richardson iteration is equivalent to pseudo-time iterations with explicit Euler, this method is also covered.


\subsection{Pseudo-time iterations in locally conservative form} \label{sec:conservative_form}

We once again consider the finite volume method \eqref{eq:FV} where the implicit Euler method is used as time discretization. In order to apply pseudo-time iterations to this scheme, we introduce a pseudo-time derivative,
\begin{equation*}
	\frac{\partial u_i}{\partial \tau} + g_i(\xvec{u}) = 0, \qquad u_i(0) = u_{0_i}, \qquad i = \dots, -1, 0, 1, \dots
\end{equation*}
where the nonlinear function $g_i$ is given by
\begin{equation} \label{eq:g-fun}
    g_i(\xvec{u}) = \frac{u_i - u_i^n}{\dt} + \frac{1}{\dx} \left( \hat{f}_{i+\frac{1}{2}}(\xvec{u}) - \hat{f}_{i-\frac{1}{2}}(\xvec{u}) \right).
\end{equation}
Several different methods are available for iterating in pseudo-time \cite{swanson2007convergence, birken2019preconditioned}. Herein, we use an explicit $s$-stage Runge-Kutta method (ERK). Let $(\tmat{A},\tvec{b},\tvec{c})$ denote the coefficient matrix and vectors of the ERK method. We denote the $k$th pseudo-time iterate by $\iter{k}$. The subsequent iterate $\iter{k+1}$ is computed from $\iter{k}$ as
\begin{equation} \label{eq:RK_step}
    \iter{k+1} = \iter{k} - \dtau_k \sum_{j=1}^s b_j g_i \left( \xvec{U}_j^{(k)} \right),
\end{equation}
where the stage vectors $\xvec{U}_j^{(k)}$, $j = 1, \dots, s$ have elements
\begin{equation} \label{eq:RK_stages}
U_{j_\iota}^{(k)} = \iter[\iota]{k} - \dtau_k \sum_{l=1}^{j-1} a_{j,l} g_\iota \left( \xvec{U}_l^{(k)} \right), \quad \iota = i-p, \dots, i+q.
\end{equation}
As previously, $p$ and $q$ determine the bandwidth of the finite volume stencil.

In the remainder we always use a fixed number $N$ of iterations. A step in physical time is taken by setting $u_i^{n+1} = \iter{N}$. Throughout, $\iter{0} = u_i^n$ is chosen as initial iterate.

Recall that the \emph{stability function} $\phi(z)$ of an RK method $(\tmat{A},\tvec{b},\tvec{c})$ is given by (see e.g. \cite[Chapter IV.3]{wanner1996solving})
\begin{equation} \label{eq:stability_function}
    \phi(z) = 1 + z \tvec{b}^\top (\tmat{I} - z \tmat{A})^{-1} \tvec{1},
\end{equation}
where $\tmat{I}$ is the $s \times s$ identity matrix. The \emph{stability region} of the RK method is defined as the subset of the complex plane for which $|\phi(z)| < 1$. The proof of the following lemma is rather lengthy and is therefore deferred to Appendix \ref{appendix:proof}.

\begin{lemma} \label{lemma:flux_form}
For each pseudo-time iteration $k = 0, \dots, N-1$, let $\mu_k = \dtau_k/\dt$ and set $\iter{0} = u_i^n$. Then, for any $N \geq 1$, the elements of the pseudo-time iterate $\iter{N}$ satisfy the relation
\begin{equation} \label{eq:pseudo-time_conservative}
    \frac{\iter{N} - u_i^n}{\dt} + \frac{1}{\dx} \left( \hat{h}_{i+\frac{1}{2}}^{(N)} - \hat{h}_{i-\frac{1}{2}}^{(N)} \right) = 0, \qquad i = \dots, -1, 0, 1, \dots
\end{equation}
where the flux $\hat{h}_{i+\frac{1}{2}}^{(N)}$ is given by
\begin{equation} \label{eq:H-flux}
    \hat{h}_{i+\frac{1}{2}}^{(N)} = \sum_{k=0}^{N-1} \mu_k \tvec{b}^\top (\tmat{I} + \mu_k \tmat{A})^{-1} \left( \prod_{l=k+1}^{N-1} \phi(-\mu_{l}) \right) \hat{\tvec{f}}_{i+\frac{1}{2}}^{(k)}
\end{equation}
and $\hat{\tvec{f}}_{i+\frac{1}{2}}^{(k)} = \left( \hat{f}_{i+\frac{1}{2}}\left( \xvec{U}_1^{(k)} \right), \dots, \hat{f}_{i+\frac{1}{2}} \left( \xvec{U}_s^{(k)} \right) \right)^\top$.
\end{lemma}

\begin{remark}
The product in \eqref{eq:H-flux} is empty when $k = N-1$. To handle this case we use the convention
$$
\prod_{l=N}^{N-1} \phi(-\mu_{l}) = 1.
$$
\end{remark}

It follows immediately from Lemma \ref{lemma:flux_form} that the locally conservative nature of the discretization \eqref{eq:FV} is preserved by the pseudo-time iterations, albeit with the new numerical flux $\hat{h}_{i+\frac{1}{2}}^{(N)}$ in place of $\hat{f}_{i+\frac{1}{2}}$. The question with which flux $\hat{h}_{i+\frac{1}{2}}^{(N)}$ is consistent, is answered in the following Theorem:

\begin{theorem} \label{thm:local_conservation}
Let $N$ be given and let $\mu_k = \dtau_k/\dt$ for $k = 0, \dots, N-1$. Set $\iter{0} = u_i^n$ and terminate the pseudo-time iteration after $N$ steps. Then the resulting scheme can be written in the locally conservative form
\begin{equation} \label{eq:conservative_form}
    \frac{u_i^{n+1} - u_i^n}{\dt} + \frac{1}{\dx} \left( \hat{h}_{i+\frac{1}{2}}^{(N)} - \hat{h}_{i-\frac{1}{2}}^{(N)} \right) = 0, \qquad i = \dots, -1, 0, 1, \dots
\end{equation}
where $\hat{h}_{i+\frac{1}{2}}^{(N)}$, given by \eqref{eq:H-flux}, is a numerical flux consistent with the flux $c(\mu_0,\dots,\mu_{N-1}) f$, with
\begin{equation} \label{eq:wave_speed}
    c(\mu_0,\dots,\mu_{N-1}) = 1 - \prod_{l=0}^{N-1} \phi(-\mu_{l}).
\end{equation}
\end{theorem}

\begin{proof}
Firstly, the locally conservative form \eqref{eq:conservative_form} follows from setting $u_i^{n+1} = \iter{N}$ in \eqref{eq:pseudo-time_conservative}. Secondly we recall that the numerical flux $\hat{f}_{i+\frac{1}{2}}$ depends on $p+q+1$ parameters, e.g. $\hat{f}_{i+\frac{1}{2}}(\xvec{w}) = \hat{f}(w_{i-p}, \dots, w_{i+q})$. Inspecting the flux $\hat{h}_{i+\frac{1}{2}}^{(N)}$ we see that it is additionally dependent on $u_i^n$ so that we may write $\hat{h}_{i+\frac{1}{2}}^{(N)}(\xvec{w}) = \hat{h}^N(w_{i-p}, \dots, w_{i+q}; u_i^n)$. To establish the consistency we must therefore show that $\hat{h}^N(u, \dots, u; u) = c f(u)$. To do this, we first note from \eqref{eq:g-fun} that $g_i(u,\dots,u; u) = 0$ due to the consistency of $\hat{f}_{i+\frac{1}{2}}$. Using the fact that $\tmat{A}$ is lower triangular for any ERK method, it follows from \eqref{eq:RK_step} and \eqref{eq:RK_stages} that $U_{j_\iota}^{(k)} = \iter{k} = u$ for every $j$, $k$ and $\iota = i-p, \dots, i+q$. Consequently, the vector $\hat{\tvec{f}}_{i+\frac{1}{2}}^{(k)}$ becomes
$$
\hat{\tvec{f}}^k(u,\dots,u;u) = \left( \hat{f}(u,\dots,u), \dots, \hat{f}(u,\dots,u) \right) = \hat{f}(u,\dots,u) \tvec{1} = f(u) \tvec{1},
$$
where the consistency of $\hat{f}$ has been used in the final equality. Inserting this into the flux function $\hat{h}_{i+\frac{1}{2}}^{(N)}$ in \eqref{eq:H-flux} and using \eqref{eq:stability_function} gives
\begin{align*}
    \hat{h}_{i+\frac{1}{2}}^{(N)} &= \left[ \sum_{k=0}^{N-1} \mu_k \tvec{b}^\top (\tmat{I} + \mu_k \tmat{A})^{-1} \tvec{1} \prod_{l=k+1}^{N-1} \phi(-\mu_{l}) \right] f(u) \\
    &= \left[ \sum_{k=0}^{N-1} [1 - \phi(-\mu_k)] \prod_{l=k+1}^{N-1} \phi(-\mu_{l}) \right] f(u) \\
    &= \sum_{k=0}^{N-1} \left[ \prod_{l=k+1}^{N-1} \phi(-\mu_{l}) - \prod_{l=k}^{N-1} \phi(-\mu_{l}) \right] f(u) \\
    &= \left[ 1 - \prod_{l=0}^{N-1} \phi(-\mu_{l}) \right] f(u) \\
    &= c(\mu_0,\dots,\mu_{N-1}) f(u),
\end{align*}
where we have utilized the fact that the sum in the third equality is telescoping. Finally, note that $\hat{h}_{i+\frac{1}{2}}^{(N)}$ is formed by a linear combination of evaluations of $\hat{f}_{i+\frac{1}{2}}$ and is therefore Lipschitz continuous.
\end{proof}

The fact that $\hat{h}_{i+\frac{1}{2}}^{(N)}$ is inconsistent with $f$ when $c \neq 1$ is a manifestation of the error associated with the pseudo-time iterations. This has important implications on the convergence of the resulting scheme. What follows is an extension of the Lax-Wendroff theorem (see \cite[Chapter 12]{leveque1992numerical}) pertaining to ERK pseudo-time iterations:

\begin{theorem} \label{thm:LW}
Consider a sequence of grids $(\Delta x_\ell, \Delta t_\ell)$ such that $\Delta x_\ell, \Delta t_\ell \rightarrow 0$ as $\ell \rightarrow \infty$. Fix $N$ independently of $\ell$, set $\iter{0} = u_i^n$ and terminate the pseudo-time iterations after $N$ steps. Let $\Delta \tau_{k,\ell} / \Delta t_\ell = \mu_{k,\ell} = \mu_k$ be constants independent of $\ell$ for each $k=0,\dots,N-1$. Suppose that the numerical flux $\hat{f}$ in \eqref{eq:FV} is consistent with $f$ and that Assumption \ref{assumptions} is satisfied. Then, $u(x,t)$ is a weak solution of the conservation law
\begin{equation} \label{eq:pseudo_conservation_law}
    u_t + c(\mu_0,\dots,\mu_{N-1}) f_x = 0, \qquad c(\mu_0,\dots,\mu_{N-1}) = 1 - \prod_{l=0}^{N-1} \phi(-\mu_{l}).
\end{equation}
\end{theorem}

\begin{proof}
We follow the proof of the Lax-Wendroff theorem given in \cite[Chapter 12]{leveque1992numerical} with changes and details added where necessary. Throughout the proof we let $\mathcal{U}_\ell(x,t)$ denote the piecewise constant function that takes the solution value $u_i^n$ in $(x_i, x_{i+1}] \times (t_{n-1}, t_n]$ on the $\ell$th grid. Similarly, for $j = 1,\dots,s$ we let $\mathcal{U}_{j_\ell}^{(k)}(x,t)$ be the piecewise constant function that takes the value $U_{j_i}^{(k)}$ in $(x_i, x_{i+1}] \times (t_{n-1}, t_n]$. 

The discretization \eqref{eq:conservative_form} can equivalently be expressed as
\begin{equation} \label{eq:FV_discretization}
    \dx_\ell \left[ \mathcal{U}_\ell(x_i,t_{n+1}) - \mathcal{U}_\ell(x_i,t_n) \right] + \dt_\ell \left[ \hat{h}_{i+\frac{1}{2}}^{(N)} - \hat{h}_{i-\frac{1}{2}}^{(N)} \right] = 0.
\end{equation}
Let $\varphi \in C_0^1$ be a compactly supported test function. Multiply \eqref{eq:FV_discretization} by $\varphi(x_i,t_n)$ and sum over all $n$ and $i$. Due to the compact support of $\varphi$, these sums can be extended arbitrarily beyond the bounds of $x$ and $t$, hence we obtain
\begin{equation} \label{eq:first_sums}
    \begin{aligned}
        \dx_\ell \sum_{n=0}^\infty \sum_{i=-\infty}^\infty & \varphi(x_i,t_n) \left( \mathcal{U}_\ell(x_i,t_{n+1}) - \mathcal{U}_\ell(x_i,t_n) \right) \\
        &+ \dt_\ell \sum_{n=0}^\infty \sum_{i=-\infty}^\infty \varphi(x_i,t_n) \left[ \hat{h}_{i+\frac{1}{2}}^{(N)} - \hat{h}_{i-\frac{1}{2}}^{(N)} \right] = 0.
    \end{aligned}  
\end{equation}
At this point we will use summation by parts on the sums in \eqref{eq:first_sums}. Recall that for two sequences $(a_i)$ and $(b_i)$, the summation by parts formula can be expressed as
$$
\sum_{j=1}^M a_j (b_j - b_{j-1}) = a_M b_M - a_1 b_0 - \sum_{j=1}^M (a_{j+1} - a_j) b_j.
$$
Applied to the $n$-sum in the first term and to the $i$-sum in the second term in \eqref{eq:first_sums}, this results in
\begin{equation} \label{eq:terms_approximating_integrals}
    \begin{aligned}
        \dx_\ell \dt_\ell & \left[ \sum_{n=1}^\infty \sum_{i=-\infty}^\infty \right. \left( \frac{\varphi(x_i,t_n) - \varphi(x_i,t_{n-1})}{\dt_\ell} \right) \mathcal{U}_\ell(x_i,t_n) \\
        &+ \left. \sum_{n=0}^\infty \sum_{i=-\infty}^\infty \left( \frac{\varphi(x_{i+1},t_n) - \varphi(x_i,t_n)}{\dx_\ell} \right) \hat{h}_{i+\frac{1}{2}}^{(N)} \right] \\
        &= -\dx_\ell \sum_{i=-\infty}^\infty \varphi(x_i,0) \mathcal{U}_\ell(x_i,0).
    \end{aligned}
\end{equation}
Here we have used the fact that $\varphi$ has compact support in order to eliminate all boundary terms except the one at $n=0$.

We now let $\ell \rightarrow \infty$ and investigate the convergence of the terms arising in \eqref{eq:terms_approximating_integrals}. The first and third terms are identical to those in the proof of the Lax-Wendroff theorem \cite[Chapter 12]{leveque1992numerical}. Thus, it can immediately be concluded that the first term converges to
$$
\int_0^\infty \int_{-\infty}^\infty \varphi_t u(x,t) \text{d}x \text{d}t
$$
and the third one to
$$
-\int_{-\infty}^\infty \varphi(x,0) u(x,0) \text{d}x.
$$
It remains to investigate the second term in \eqref{eq:terms_approximating_integrals}.

Expanding $\hat{h}_{i+\frac{1}{2}}^{(N)}$ using \eqref{eq:H-flux} yields for the second term in \eqref{eq:terms_approximating_integrals}
\begin{align*}
\dx_\ell & \dt_\ell \sum_{n=0}^\infty \sum_{i=-\infty}^\infty \left( \frac{\varphi(x_{i+1},t_n) - \varphi(x_i,t_n)}{\dx_\ell} \right) \\
& \sum_{k=0}^{N-1} \mu_k \tvec{b}^\top (\tmat{I} + \mu_k \tmat{A})^{-1} \left( \prod_{l=k+1}^{N-1} \phi(-\mu_{l}) \right) \hat{\tvec{f}}_{i+\frac{1}{2}}^{(k)}.
\end{align*}
To this expression, add and subtract
\begin{align*}
\dx_\ell & \dt_\ell \sum_{n=0}^\infty \sum_{i=-\infty}^\infty \left( \frac{\varphi(x_{i+1},t_n) - \varphi(x_i,t_n)}{\dx_\ell} \right) \\
& \sum_{k=0}^{N-1} \mu_k \tvec{b}^\top (\tmat{I} + \mu_k \tmat{A})^{-1} \left( \prod_{l=k+1}^{N-1} \phi(-\mu_{l}) \right) f(\mathcal{U}_\ell(x_i,t_n)) \tvec{1}
\end{align*}
in order to obtain
\begin{equation*}
    \begin{alignedat}{3}
    & \dx_\ell && \dt_\ell \sum_{n=0}^\infty \sum_{i=-\infty}^\infty \left( \frac{\varphi(x_{i+1},t_n) - \varphi(x_i,t_n)}{\dx_\ell} \right) \\
    &&& \left[ \sum_{k=0}^{N-1} \mu_k \tvec{b}^\top (\tmat{I} + \mu_k \tmat{A})^{-1} \tvec{1} \left( \prod_{l=k+1}^{N-1} \phi(-\mu_{l}) \right) \right] f(\mathcal{U}_\ell(x_i,t_n)) \\
    + & \dx_\ell && \dt_\ell \sum_{n=0}^\infty \sum_{i=-\infty}^\infty \left( \frac{\varphi(x_{i+1},t_n) - \varphi(x_i,t_n)}{\dx_\ell} \right) \\
    &&& \sum_{k=0}^{N-1} \mu_k \tvec{b}^\top (\tmat{I} + \mu_k \tmat{A})^{-1} \left( \prod_{l=k+1}^{N-1} \phi(-\mu_{l}) \right) \left( \hat{\tvec{f}}_{i+\frac{1}{2}}^{(k)} - f(\mathcal{U}_\ell(x_i,t_n)) \tvec{1} \right).
    \end{alignedat}
\end{equation*}
The bracketed part of this expression evaluates to $c(\mu_0,\dots,\mu_{N-1})$, as in the proof of Theorem \ref{thm:local_conservation}. Thus, the first half of this expression equals
$$
c(\mu_0,\dots,\mu_{N-1}) \dx_\ell \dt_\ell \sum_{n=0}^\infty \sum_{i=-\infty}^\infty \left( \frac{\varphi(x_{i+1},t_n) - \varphi(x_i,t_n)}{\dx_\ell} \right) f(\mathcal{U}_\ell(x_i,t_n)).
$$
Apart from the factor $c(\mu_0,\dots,\mu_{N-1})$, which is independent of $\ell$, this term appears identically in the proof of the Lax-Wendroff theorem \cite[Chapter 12]{leveque1992numerical} and converges to
$$
c(\mu_0,\dots,\mu_{N-1}) \int_0^\infty \int_{-\infty}^\infty \varphi_x f \text{d}x \text{d}t.
$$

It remains to show that 
\begin{equation} \label{eq:term_to_vanish}
    \begin{aligned}
    \dx_\ell & \dt_\ell \sum_{n=0}^\infty \sum_{i=-\infty}^\infty \left( \frac{\varphi(x_{i+1},t_n) - \varphi(x_i,t_n)}{\dx_\ell} \right) \\
    & \sum_{k=0}^{N-1} \mu_k \tvec{b}^\top (\tmat{I} + \mu_k \tmat{A})^{-1} \left( \prod_{l=k+1}^{N-1} \phi(-\mu_{l}) \right) \left( \hat{\tvec{f}}_{i+\frac{1}{2}}^{(k)} - f(\mathcal{U}_\ell(x_i,t_n)) \tvec{1} \right)
    \end{aligned}
\end{equation}
vanishes in the limit $\ell \rightarrow \infty$. Since $\mathcal{U}_\ell$ and $\mathcal{U}_{j_\ell}^{(k)}$ are both constant in $(x_i,x_{i+1}] \times (t_{n-1},t_n]$, we can rewrite \eqref{eq:term_to_vanish} as
\begin{align*}
\sum_{n=0}^\infty & \sum_{i=-\infty}^\infty \int_{x_i}^{x_{i+1}} \int_{t_{n-1}}^{t_n} \varphi_x(x,t_n) \\
& \sum_{k=0}^{N-1} \mu_k \tvec{b}^\top (\tmat{I} + \mu_k \tmat{A})^{-1} \left( \prod_{l=k+1}^{N-1} \phi(-\mu_{l}) \right) \left( \hat{\tvec{f}}_{i+\frac{1}{2}}^{(k)}(x,t) - f(\mathcal{U}_\ell(x,t)) \tvec{1} \right).
\end{align*}
Here, $\hat{\tvec{f}}_{i+\frac{1}{2}}^{(k)}(x,t)$ is placeholder notation for the vector whose $j$th element is
$$
\hat{f}_{i+\frac{1}{2}} \left( \mathcal{U}_{j_\ell}^{(k)}(x-p\dx,t), \dots, \mathcal{U}_{j_\ell}^{(k)}(x+q\dx,t) \right).
$$
To establish that this term indeed vanishes, it suffices to show that
\begin{equation} \label{eq:term_to_be_bounded}
    \left| \sum_{k=0}^{N-1} \mu_k \tvec{b}^\top (\tmat{I} + \mu_k \tmat{A})^{-1} \left( \prod_{l=k+1}^{N-1} \phi(-\mu_{l}) \right) \left( \hat{\tvec{f}}_{i+\frac{1}{2}}^{(k)}(x,t) - f(\mathcal{U}_\ell(x,t)) \tvec{1} \right) \right|
\end{equation}
tends to zero for almost every $x$. By the Cauchy-Schwarz and triangle inequalities, \eqref{eq:term_to_be_bounded} is bounded by
\begin{equation} \label{eq:next_term_to_be_bounded}
\sum_{k=0}^{N-1} \| \mu_k \tvec{b}^\top (\tmat{I} + \mu_k \tmat{A})^{-1} \| \left( \prod_{l=k+1}^{N-1} | \phi(-\mu_{l}) | \right) \left\| \hat{\tvec{f}}_{i+\frac{1}{2}}^{(k)}(x,t) - f(\mathcal{U}_\ell(x,t)) \tvec{1} \right\|,
\end{equation}
where the Euclidean norm in $\mathbb{R}^s$ is used. The only term in \eqref{eq:next_term_to_be_bounded} that depends on $\ell$ is $\left\| \left( \hat{\tvec{f}}_{i+\frac{1}{2}}^{(k)}(x,t) - f(\mathcal{U}_\ell(x,t)) \tvec{1} \right) \right\|$ and it therefore suffices to show that this term vanishes for almost every $x$.

Recall that $f(\mathcal{U}_\ell) = \hat{f}_{i+\frac{1}{2}}(\mathcal{U}_\ell, \dots, \mathcal{U}_\ell)$ by consistency. A standard norm inequality gives
\begin{align*}
    & \left\| \hat{\tvec{f}}_{i+\frac{1}{2}}^{(k)}(x,t) - f(\mathcal{U}_\ell(x,t)) \tvec{1} \right\| \\
    \leq \sqrt{s} \max_{1 \leq j \leq s} &\left| \hat{f}_{i+\frac{1}{2}}\left( \mathcal{U}_{j_\ell}^{(k)}(x-p\dx,t), \dots, \mathcal{U}_{j_\ell}^{(k)}(x+q\dx,t) \right) \right. \\
    &\left. -\hat{f}_{i+\frac{1}{2}}(\mathcal{U}_\ell(x,t), \dots, \mathcal{U}_\ell(x,t)) \right| \\
    \leq \sqrt{s} \max_{1 \leq j \leq s} &\left\{ \left| \hat{f}_{i+\frac{1}{2}} \left( \mathcal{U}_{j_\ell}^{(k)}(x-p\dx,t), \dots, \mathcal{U}_{j_\ell}^{(k)}(x+q\dx,t) \right) \right. \right. \\
    &- \left. \hat{f}_{i+\frac{1}{2}}(\mathcal{U}_\ell(x-p\dx,t), \dots, \mathcal{U}_\ell(x+q\dx,t)) \right| \\
    + & \left| \hat{f}_{i+\frac{1}{2}}(\mathcal{U}_\ell(x-p\dx,t), \dots, \mathcal{U}_\ell(x+q\dx,t)) \right. \\
    &- \left. \left. \hat{f}_{i+\frac{1}{2}}(\mathcal{U}_\ell(x,t), \dots, \mathcal{U}_\ell(x,t)) \right| \right\}.
\end{align*}
Since $\hat{f}_{i+\frac{1}{2}}$ is Lipschitz continuous in each argument there is a constant $L$ such that the final expression is bounded by
\begin{align*}
    & L\sqrt{s} \max_{1 \leq j \leq s} \max_{p \leq r \leq q} \left| \mathcal{U}_{j_\ell}^{(k)}(x+r\dx,t) - \mathcal{U}_\ell(x+r\dx,t) \right| \\
    + & L \sqrt{s} \max_{p \leq r \leq q} \left| \mathcal{U}_\ell(x+r\dx,t) - \mathcal{U}_\ell(x,t) \right|.
\end{align*}
The second of these terms appear in the proof of the Lax-Wendroff theorem \cite[Chapter 12]{leveque1992numerical} and vanishes in the limit for almost every $x$ due to the bounded total variation of $\mathcal{U}_\ell$. Further, from \eqref{eq:RK_stages} and \eqref{eq:RK_step} and the fact that $\iter{0} = u_i^n$ it follows that for every $j=1,\dots,s$ and $k \geq 0$, $U_{j_\iota}^{(k)} = u_\iota^n$ for each $\iota = i-p, \dots, i+q$ in the limit of vanishing $\dtau_\ell$. Consequently, $\left| \mathcal{U}_{j_\ell}^{(k)}(\cdot,t) - \mathcal{U}_\ell(\cdot,t) \right|$ vanishes identically as $\ell \rightarrow \infty$.
\end{proof}

A few remarks about Theorem \ref{thm:LW} are in place: First, we demand from a useful iterative method that it converges to the correct solution as $N \rightarrow \infty$. Thus, the pseudo-time steps should be chosen in a way that ensures that $c \rightarrow 1$, or equivalently,
\begin{equation} \label{eq:RK_stability_requirement}
\prod_{l=0}^{N-1} \phi(-\mu_{l}) \rightarrow 0 \quad \text{as} \quad N \rightarrow \infty.
\end{equation}
The simplest way to do this is to choose each pseudo-time step so that $|\phi(-\mu_{l})| < 1$, i.e. to stay within the stability region of the RK method.

Secondly, observe that if $\phi(-\mu_{l}) = 0$ for any $l$, then $c=1$ irrespective of how many further iterations that are carried out. For some RK methods such a root exists while for others it does not. For instance, the explicit Euler method has stability function $\phi(-\mu) = 1 - \mu$, hence $\mu = 1$ is a root of $\phi$. On the other hand, Heun's method has stability function $\phi(-\mu) = 1 - \mu + \mu^2 /2 > 0$ for all $\mu \in \mathbb{R}$ and therefore does not have any real roots. A strategy is thus to choose a RK method with a root, begin the pseudo-time iterations with a step that corresponds to this root, then resort to a conventional method for choosing the remaining pseudo-time steps. The initial iteration will not change the limit as $N \rightarrow \infty$ if the remaining pseudo-time steps are selected within the stability region.


\subsection{Numerical results}

Next we validate Theorem \ref{thm:LW} by numerically solving a series of linear and nonlinear conservation laws.

\subsubsection{Linear advection}

The first setting is the linear advection equation \eqref{eq:advection_equation}. The computational domain is $x \in (-1,1]$, $t \in (0,0.25]$ and periodic boundary conditions are used. The upwind flux $\hat{f}_{i+\frac{1}{2}} = u_{i}^{n+1}$ is used for the spatial discretization. The resulting finite volume method becomes
$$
\frac{u_i^{n+1} - u_i^n}{\dt} + \frac{1}{\dx} \left( u_i^{n+1} - u_{i-1}^{n+1} \right) = 0, \qquad i = 1, \dots, m.
$$
Throughout the experiments the temporal and spatial increments are chosen to be equal: $\dt = \dx$.

We validate Theorem \ref{thm:LW} by studying the convergence of the numerical scheme to the solution of the original advection problem \eqref{eq:advection_equation} as well as to the solution of the modified version $u_t + c u_x = 0$. As pseudo-time iteration we use the explicit Euler method, Heun's method and the third order strong stability preserving RK method SSPRK3 \cite{shu1988efficient}. Theorem \ref{thm:LW} predicts that these methods respectively will modify the propagation speed by the factor
\begin{align*}
c(\mu_0,\dots,\mu_{N-1}) &= 1 - \prod_{l=0}^{N-1} (1-\mu_{l}), \\
c(\mu_0,\dots,\mu_{N-1}) &= 1 - \prod_{l=0}^{N-1} \left( 1 - \mu_{l} + \frac{\mu_{l}^2}{2} \right), \\
c(\mu_0,\dots,\mu_{N-1}) &= 1 - \prod_{l=0}^{N-1} \left( 1 - \mu_{l} + \frac{\mu_{l}^2}{2} - \frac{\mu_{l}^3}{6} \right).
\end{align*}
Here, we fix $N=4$ and consider two different sequences of pseudo-time steps, one with constant and one with variable step sizes. The first is given by $\mu_{l} = 1/20$ and the second by $\mu_{l} = 2^{-l}$ for $l = 0, \dots, 3$. The corresponding modification constants $c(\mu_0, \dots, \mu_3)$ are shown in Table \ref{tab:c}. Note that with the second sequence, $c=1$ for the explicit Euler method. This is due to the fact that $\mu_0$ is a root of the stability polynomial in this case.

\begin{table}[h!]
\begin{center}
\caption{Modification constant $c(\mu_0, \dots, \mu_3)$ for different ERK methods and choices of pseudo-time steps}
\begin{tabular}{ |c|c|c|c| } 
 \hline 
 & Euler & Heun & SSPRK3 \\
 \hline
 $\mu_l = \frac{1}{20}$ & $0.1855$ & $0.1812$ & $0.1813$ \\
 \hline 
 $\mu_l = 2^{-l}$ & $1$ & $0.7845$ & $0.8616$ \\
 \hline
\end{tabular}
\label{tab:c}
\end{center}
\end{table}

The advection problem is solved on a sequence of grids with grid spacing $\dx = 2 / (40 \times 2^j)$ for $j = 1, \dots, 12$. The $L^2$ error is calculated with respect to the exact solution of the original and modified equations. The results are shown in Fig. \ref{fig:advection_convergence}. When constant pseudo-time steps are used (Fig. \ref{fig:advection_convergence_constant}), the lines overlap. None of the methods converge to the solution of the original problem since the numerical solution is moving with the incorrect speed. Instead, all three schemes converge to the solution of their respective modified conservation law. Similar results are seen when variable pseudo-time steps are used (Fig. \ref{fig:advection_convergence_variable}).

\begin{figure}[t!]
    \centering
    \begin{subfigure}{0.49\textwidth}
    \includegraphics[width=\textwidth]{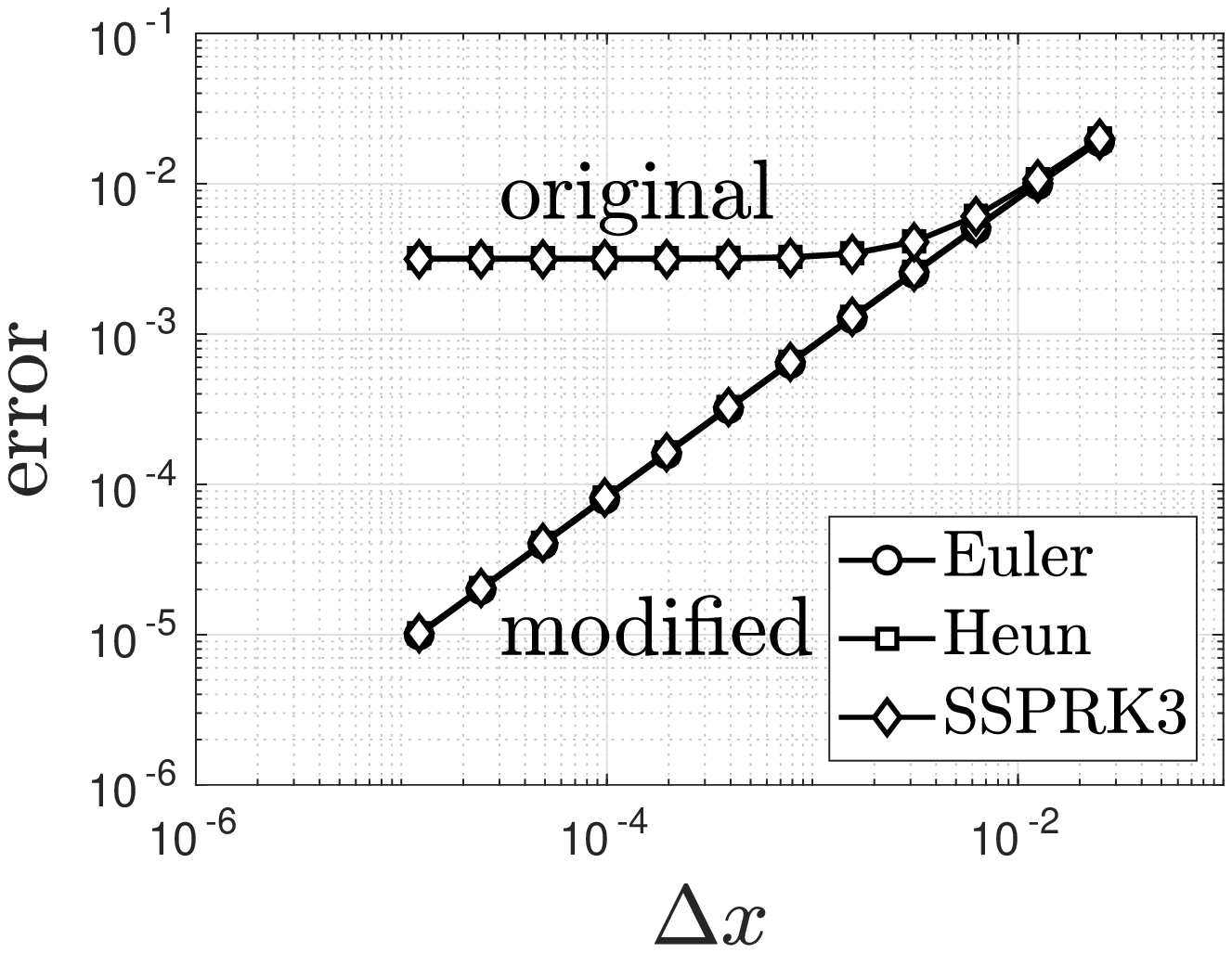}
    \caption{Constant pseudo-time steps, $\mu_l = \frac{1}{20}$.}
    \label{fig:advection_convergence_constant}
    \end{subfigure}
    \begin{subfigure}{0.49\textwidth}
    \includegraphics[width=\textwidth]{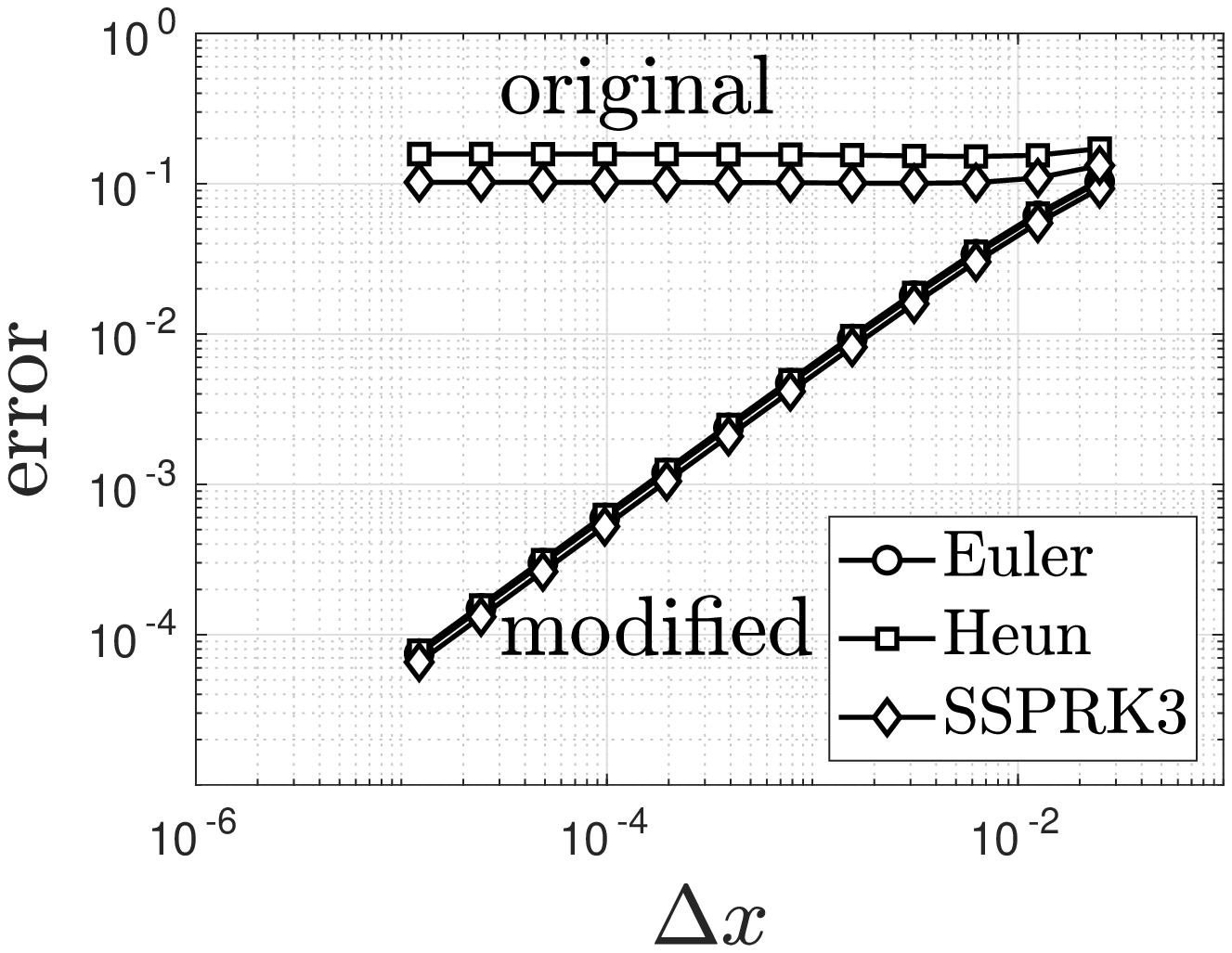}
    \caption{Variable pseudo-time steps, $\mu_l = 2^{-l}$.}
    \label{fig:advection_convergence_variable}
    \end{subfigure}
    \caption{$L^2$-errors vs $\dx$ with respect to the exact solution of the original and modified advection equations.}
    \label{fig:advection_convergence}
\end{figure}

Solutions of the advection equation \eqref{eq:advection_equation} propagate from left to right at unit speed. In this setting it follows from Theorem \ref{thm:LW} that the numerical solution will propagate with the speed $c(\mu_0,\dots,\mu_{N-1})$ in the limit $\dx, \dt \rightarrow 0$. Thus, if $c \neq 1$ in each physical time step, the numerical solution will drift out of phase. However, it should be noted that Theorem \ref{thm:LW} is an asymptotic result and does not necessarily imply that $c(\mu_0, \dots, \mu_{N-1})$ alone captures the propagation speed error on coarse grids. In fact, we may generally expect a contribution to the speed from dispersion errors built into the discretization; see e.g. \cite{linders2015uniformly,tam1993dispersion,trefethen1982group, linders2016summation, linders2017summation, linders2020accurate} for details and remedies.

In the next experiment we verify that $c$ depends on $N$ as predicted by Theorem \ref{thm:LW} by measuring the propagation speed $\bar{c}$ of the numerical solution while varying $N$. To measure $\bar{c}$ we track the $x$-coordinate of the maximum of the propagating pulse through time. To get accurate measurements we extend the computational domain to $x \in (-1/5, 1/5]$, $t \in (0,6]$.

For this experiment we fix $\dx = 0.003$ and set $\mu_{l} = \mu$ to be fixed for each $l=0,\dots,N-1$. The measured propagation speed error $1 - \bar{c}$ for different $\mu$ are shown in Fig. \ref{fig:Heun_speed_vs_iterations} for Heun's method and in Fig. \ref{fig:SSP3_speed_vs_iterations} for SSPRK3. The solid lines show the theoretically predicted error,
$$
1 - c(\mu_0,\dots,\mu_{N-1}) = \phi(-\mu)^N.
$$
The measured and theoretical speed errors agree well for $\mu = 0.05$ and  $\mu = 0.2$. For $\mu = 0.5$ a discrepancy between theory and measurement is seen for small errors. This suggests that the dispersion error intrinsic to the finite volume scheme is starting to dominate the propagation speed error.

\begin{figure}[t!]
    \centering
    \begin{subfigure}{0.49\textwidth}
    \includegraphics[width=\textwidth]{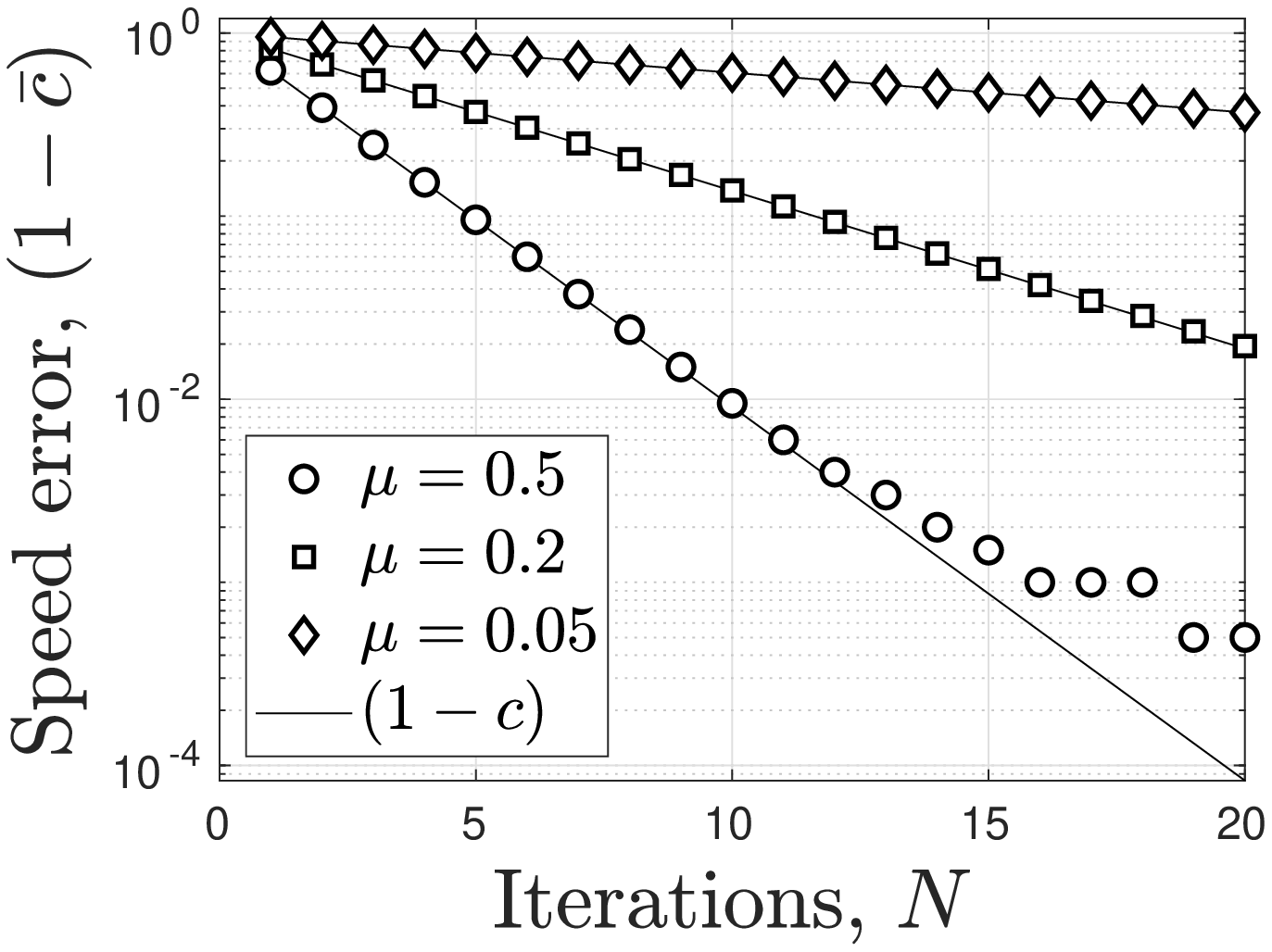}
    \caption{Heun's method.}
    \label{fig:Heun_speed_vs_iterations}
    \end{subfigure}
    \begin{subfigure}{0.49\textwidth}
    \includegraphics[width=\textwidth]{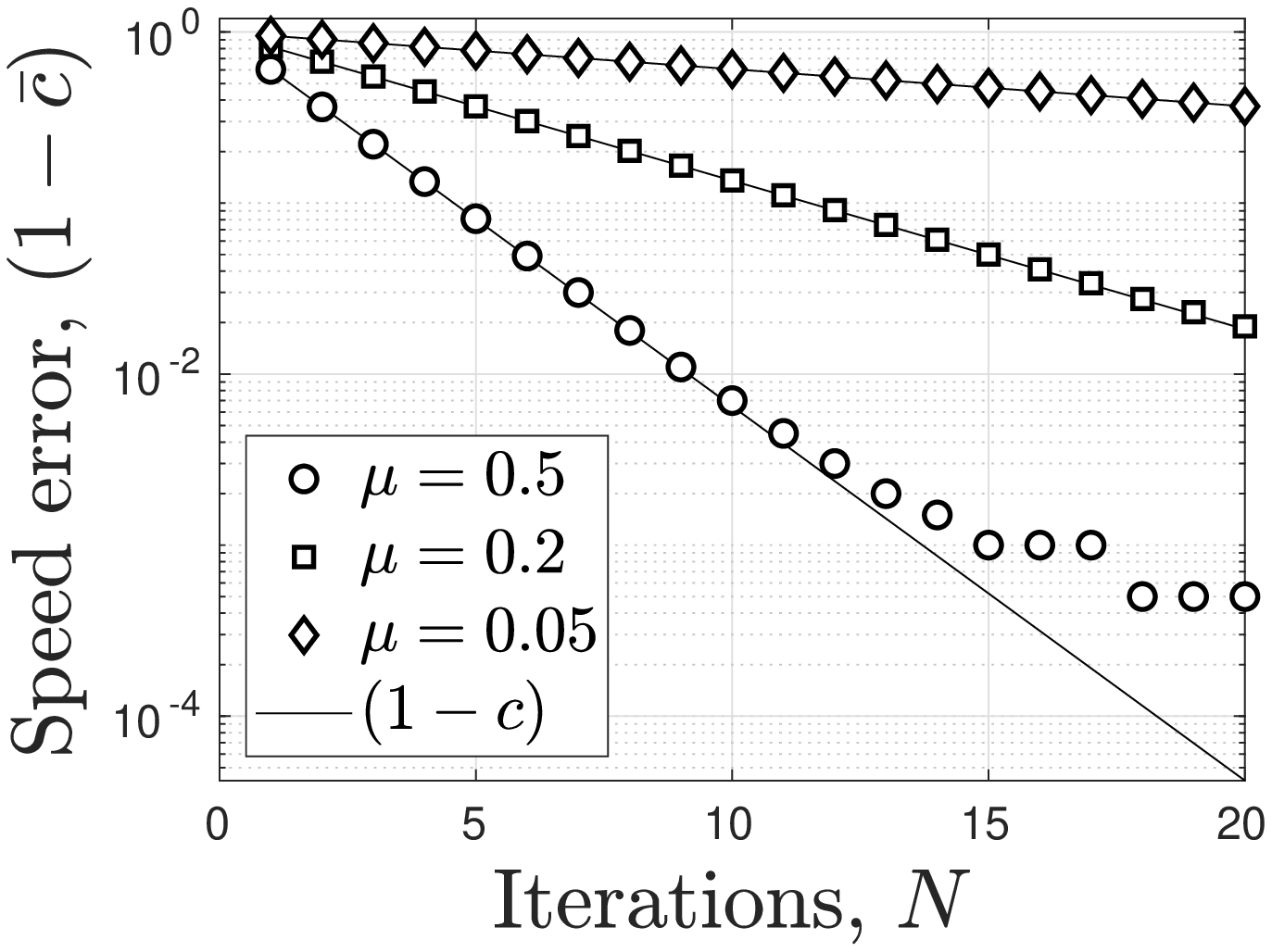}
    \caption{SSPRK3.}
    \label{fig:SSP3_speed_vs_iterations}
    \end{subfigure}
    \caption{Proparagion speed error vs number of iterations per time step for the linear advection problem.}
    \label{fig:speed_vs_iterations}
\end{figure}


\subsection{Burgers' equation}

Next we consider a triangular shock wave propagating under the 1D Burgers' equation with periodic boundary conditions:
\begin{equation} \label{eq:Burger_shock_triangle}
\begin{aligned}
u_t + \left( \frac{u^2}{2} \right)_x &= 0, \quad x \in (0, 1], \\
u(x,0) &=
\begin{cases}
x & \text{ if } x \leq 0.5, \\
0 & \text{ otherwise}.
\end{cases}
\end{aligned}
\end{equation}
The exact solution to this problem is given by
$$
u(x,t) =
\begin{cases}
\frac{x}{t+1} & \text{ if } x \leq \frac{1}{2} \sqrt{t+1}, \\
0 & \text{ otherwise}.
\end{cases}
$$
Modifying the conservation law to $u_t + (c u^2/2)_x = 0$ while retaining the same initial condition modifies the exact solution to
$$
u(x,t) =
\begin{cases}
\frac{x}{ct+1} & \text{ if } x \leq \frac{1}{2} \sqrt{ct+1}, \\
0 & \text{ otherwise}.
\end{cases}
$$
For this problem we therefore expect that pseudo-time iterations to affect both the speed and the amplitude of the shock front.

As for the advection problem, we investigate the $L^2$ convergence of the numerical scheme to the original and modified conservation laws. We once again use an upwind numerical flux and implicit Euler in time;
\begin{equation} \label{eq:Burgers_discretization}
\frac{u_i^{n+1} - u_i^n}{\dt} + \frac{1}{\dx} \left( \frac{u_i^{n+1}}{2} - \frac{u_{i-1}^{n+1}}{2} \right) = 0, \qquad i = 1, \dots, m.
\end{equation}
We run the simulation to time $t=0.1$ with $\dt = \dx$ using the sequence of grids $\dx = 1/(25 \times 2^j)$ for $j=1,\dots,15$. The explicit Euler method is used as pseudo-time iterator with $N=12$ and $\mu_l = 1/4$ for $l=0,\dots,11$. The corresponding modification constant is $c \approx 0.9683$. Fig. \ref{fig:shock_convergence_triangle} shows that the numerical solution converges to the solution of the modified conservation law as expected.

Fixing $\dx = 0.004$ and setting $\mu_l = \mu = 1/4$, we run the simulation to time $t=1$ with different choices of $N$. Fig. \ref{fig:shock_locations_triangle} shows the numerical solutions together with the initial data (dashed line) and the exact solution (dotted line). The locations of the tips of the shock waves as predicted by Theorem \ref{thm:LW} are indicated as crosses in the figure. There is good agreement between theory and experiment, although the shocks appear slightly smeared due to the built-in dissipation in the numerical scheme.

\begin{figure}[t!]
    \centering
    \begin{subfigure}{0.49\textwidth}
    \includegraphics[width=\textwidth]{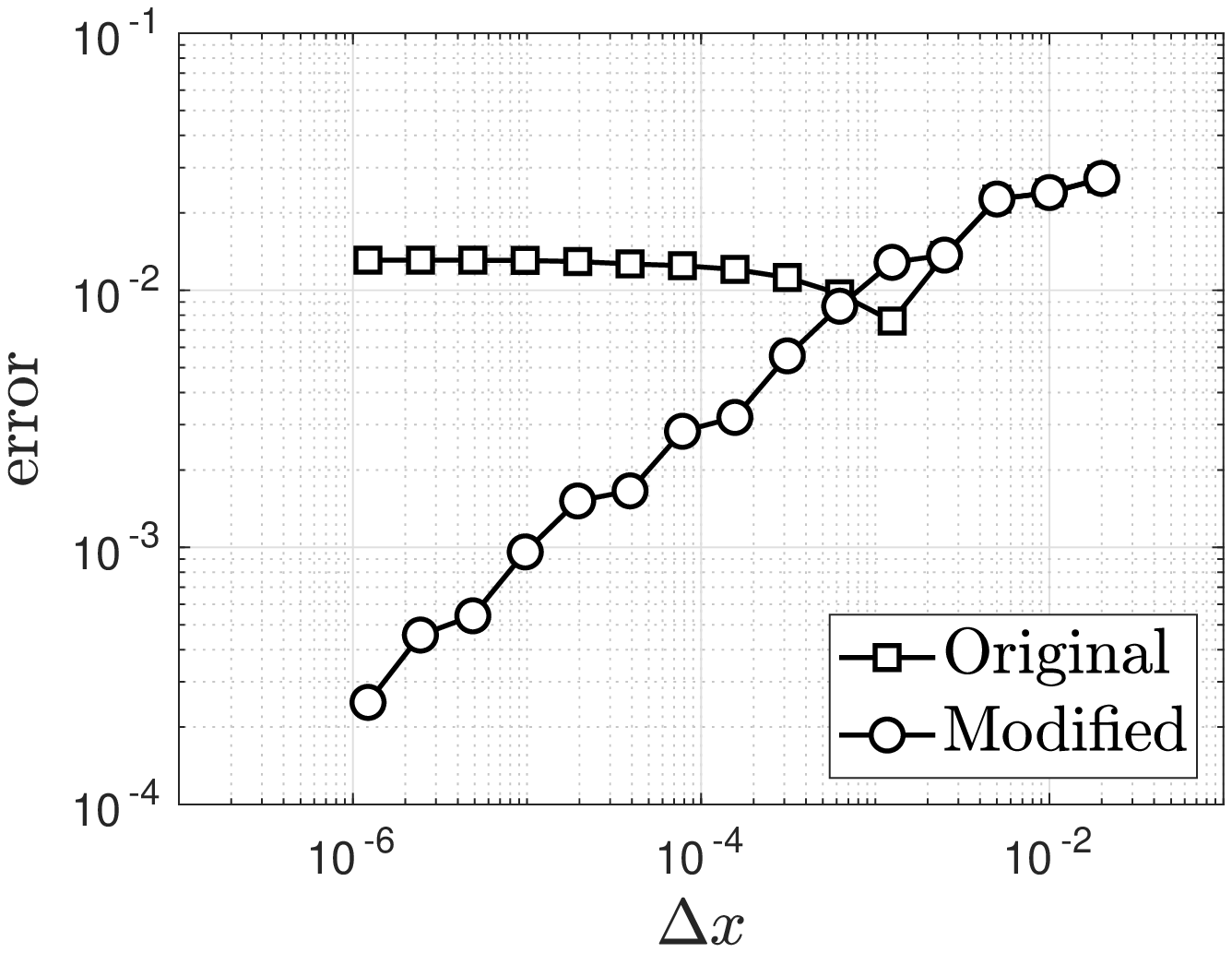}
    \caption{Triangular shock.}
    \label{fig:shock_convergence_triangle}
    \end{subfigure}
    \begin{subfigure}{0.49\textwidth}
    \includegraphics[width=\textwidth]{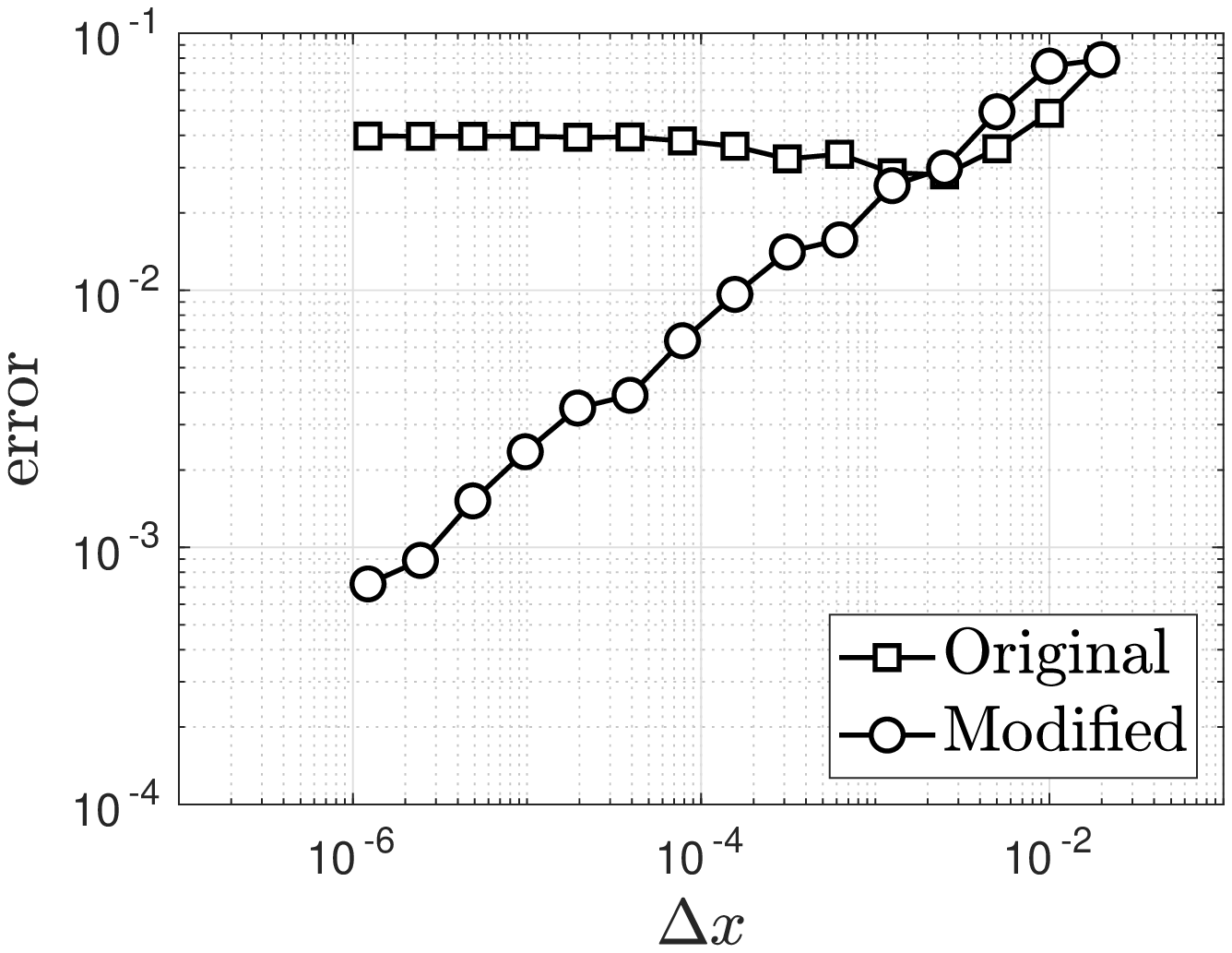}
    \caption{Step function.}
    \label{fig:shock_convergence_step}
    \end{subfigure}
    \caption{$L^2$-errors vs $\dx$ with respect to the exact solution of the original and modified Burgers' equations for (a) the triangular shock and (b) the step function.}
    \label{fig:shock_convergence}
\end{figure}

\begin{figure}[t!]
    \centering
    \includegraphics[width=0.9\textwidth]{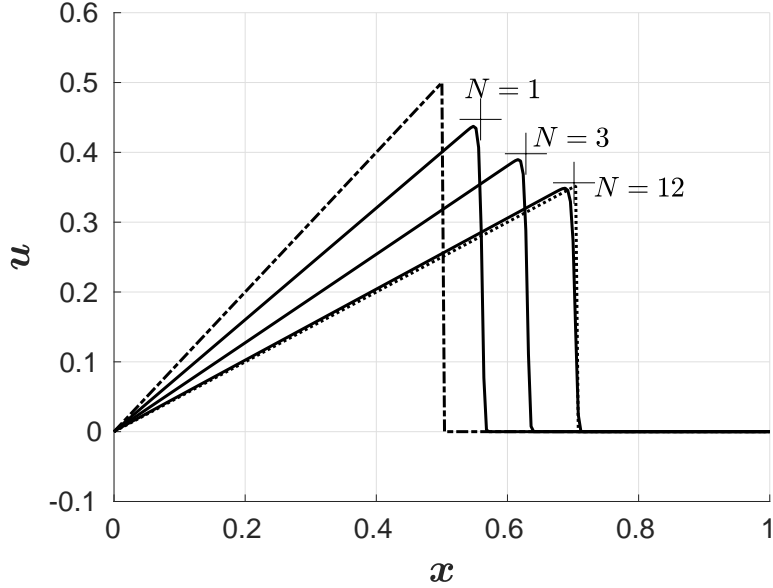}
    \caption{Numerical solutions and predicted shock locations (crosses) for the solution of Burgers' equation \eqref{eq:Burgers_equation} using different numbers of pseudo-time iterations, $N$. The dotted line indicates the exact solution. The dashed line shows the initial data.}
    \label{fig:shock_locations_triangle}
\end{figure}

Next, we repeat the experiments but change the initial data to a step function,
$$
u(x,0) =
\begin{cases}
1 \text{ if } x \leq 0.24, \\
0 \text{ otherwise}.
\end{cases}
$$
Since this data is not periodic, we impose the boundary condition $u(0,t) = 1$. Theorem \ref{thm:LW} does not treat boundary conditions and it is therefore interesting to see if it still provides useful predictions of the behaviour of the numerical solution in this setting.

The exact solution is the initial step function travelling to the right. The shock speed is given by the Rankine-Hugoniot condition as
$$
s = \frac{\frac{u_l^2}{2} - \frac{u_r^2}{2}}{u_l - u_r} = \frac{1}{2},
$$
where $u_{l,r}$ denote left and right states of the discontinuity respectively. The shock speed of the modified conservation law is instead $c/2$.

We use the exact same grids and pseudo-time iterations as for the triangular shock and measure the $L^2$ error with respect to the exact and modified conservation laws. The results are shown in Fig. \ref{fig:shock_convergence_step}. Convergence is once again seen towards the modified equation, thus verifying that Theorem \ref{thm:LW} predicts the propagation speed error correctly despite the added boundary condition.

Next we extend the time domain to $t \in (0,1]$, fix $\dx = \dt = 4 \dtau = 1/100$ and vary the number of iterations, $N$. The computed solutions using $N = 1$, $N = 3$ and $N = 12$ are shown in Fig. \ref{fig:shock_locations} together with the predicted shock locations (dashed lines). Once again, there is good agreement between prediction and experiment, although numerical dissipation smears the shock fronts somewhat.

\begin{figure}[t!]
    \centering
    \includegraphics[width=0.9\textwidth]{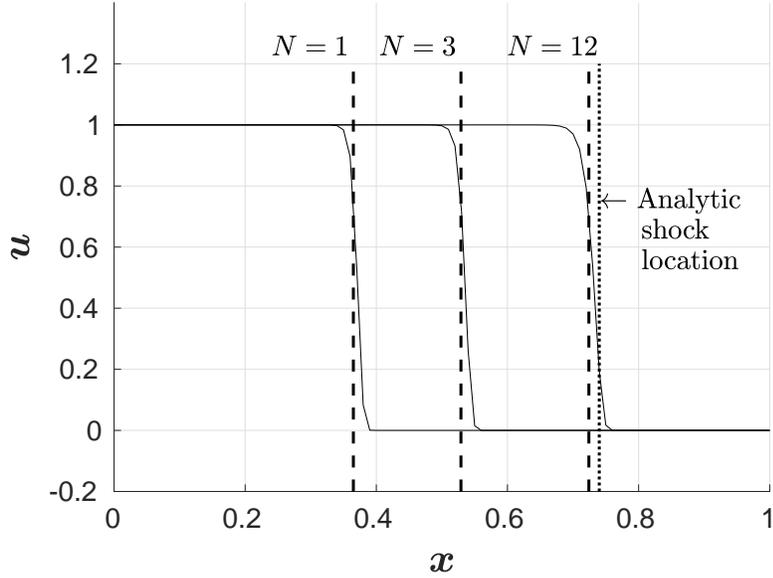}
    \caption{Numerical solutions and predicted shock locations (dashed lines) for the solution of Burgers' equation \eqref{eq:Burgers_equation} using different numbers of pseudo-time iterations, $N$.}
    \label{fig:shock_locations}
\end{figure}

As mentioned previously, the shock speed error can be eliminated entirely for the explicit Euler method by noting that $\mu = 1$ is a root of the stability polynomial $1 - \mu$. To highlight the effect of this, we introduce two strategies for choosing the pseudo-time steps:
\begin{description}
\item[Strategy 1:] Use $N=12$ iterations with $\dtau_{0,\dots,11} = \dt/4$.
\item[Strategy 2:] Use $N=9$ iterations with $\dtau_0 = \dt$ and $\dtau_{1,\dots,8} = \dt/4$.
\end{description}
The first strategy is the same that we have used in the experiments so far. The second one ensures that $c=1$ by taking a large initial pseudo-time step. Note that both strategies correspond to integration in pseudo-time to the same point; $\tau = 3 \dt$.

The relative residuals of the pseudo-time iterates in the first physical time step are shown in Fig. \ref{fig:shock_residual} for the two strategies. The large initial pseudo-time step in {\bf Strategy 2} results in a considerably greater residual reduction than the corresponding iteration using {\bf Strategy 1}. Interestingly, subsequent iterations yield faster convergence of the residual using {\bf Strategy 2} as seen by the steeper gradient. This suggests that the incorrect shock speed makes the dominant and slowest converging contribution to the residual for this problem when using {\bf Strategy 1}. We also conclude that the point to which we march in pseudo-time, here $\tau = 3 \dt$, have less of an impact on the convergence than the choice of pseudo-time steps used to reach this point.

\begin{figure}[t!]
    \centering
    \includegraphics[width=0.9\textwidth]{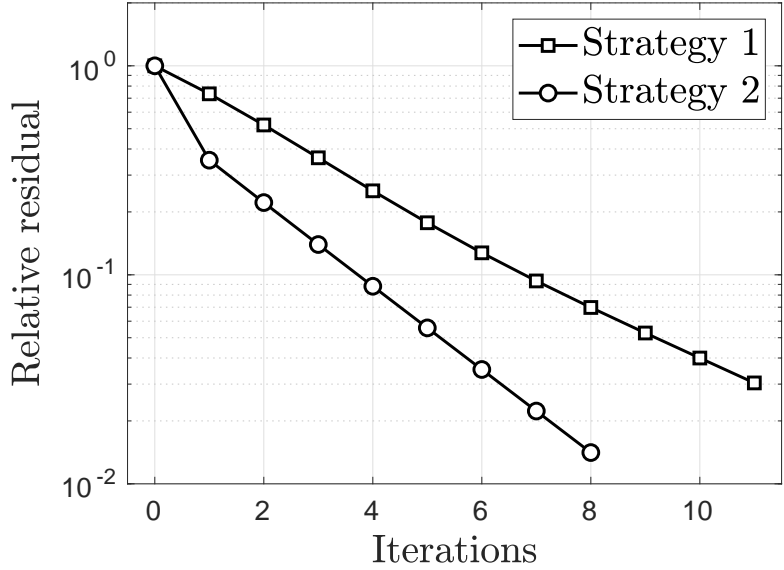}
    \caption{Relative residuals of pseudo-time iterates applied to the discretization \eqref{eq:Burgers_discretization} of Burgers' equation using {\bf Strategy 1} and {\bf Strategy 2}.}
    \label{fig:shock_residual}
\end{figure}


\subsubsection{The Euler equations}

As a second nonlinear problem we consider the 2D compressible Euler equations,
\begin{equation}
\begin{bmatrix}
\rho \\ \rho u \\ \rho v \\ \rho E
\end{bmatrix}_t
+
\begin{bmatrix}
\rho u \\
\rho u^2 + p \\
\rho u v \\
(\rho E + p) u
\end{bmatrix}_x
+
\begin{bmatrix}
\rho v \\
\rho u v \\
\rho v^2 + p \\
(\rho E + p) v
\end{bmatrix}_y
= 0,
\end{equation}
posed on the domain $(x,y) \in (-5, 15] \times (-5, 5]$. Here, $\rho, u, v, E$ and $p$ respectively denote density, horizontal and vertical velocity components, total energy per unit mass and pressure. The pressure is related to the other variables through the equation of state
$$
p = (\gamma-1) \rho e,
$$
where $\gamma = 1.4$ and $e = E - (u^2 + v^2)/2$ is the internal energy density. The domain is taken to be periodic in both spatial coordinates. The setting is the isentropic vortex problem \cite{shu1998essentially} with initial conditions
\begin{align*}
\rho_0 &= \left( \frac{1 - \epsilon^2 (\gamma - 1) M_\infty^2}{8 \pi^2} \exp{(r)} \right)^\frac{1}{\gamma - 1}, \\
u_0 &= 1 - \frac{\epsilon y}{2 \pi} \exp{(r/2)}, \\
v_0 &= \frac{\epsilon x}{2 \pi} \exp{(r/2)}, \\
p_0 &= \frac{\rho_0^\gamma}{\gamma M_\infty^2},
\end{align*}
where $r = 1 - x^2 - y^2$. Here, $\epsilon = 5$ is the circulation and $M_\infty = 0.5$ is the Mach number. As the solution evolves in time, the initial vortex propagates in the horizontal direction with unit speed.

As in previous experiments, we use implicit Euler in time, yielding a finite volume scheme of the form
\begin{equation} \label{eq:2D_discretization}
\frac{u_i^{n+1} - u_i^n}{\dt} + \frac{1}{\dx} \left( \hat{f}_{i+\frac{1}{2},j}^{n+1} - \hat{f}_{i-\frac{1}{2},j}^{n+1} \right) + \frac{1}{\dy} \left( \hat{f}_{i,j+\frac{1}{2}}^{n+1} - \hat{f}_{i,j-\frac{1}{2}}^{n+1} \right) = 0.
\end{equation}
Along the $x$-coordinate we use a fourth order centered flux
$$
\hat{f}_{i+\frac{1}{2},j} = -\frac{1}{12} f_{i-1,j} + \frac{7}{12} f_{i,j} + \frac{7}{12} f_{i+1,j} - \frac{1}{12} f_{i+2,j}
$$
and similarly along the $y$-coordinate. With this choice, the resulting problem violates the assumptions of Theorem \ref{thm:LW} in three ways: (i) It is 2D, (ii) it is a system of equations, and (iii) the scheme is not total variation bounded.

The exact solution of the isentropic vortex problem is given by the initial data (centred at the origin) translated to the point $(x,y) = (T,0)$, where $T$ is the end point of the time domain. However, the exact solution of the modified conservation law is instead centred at $(x,y) = (cT,0)$.

The explicit Euler method is again used as pseudo-time iteration. We study the convergence of the numerical solution to the exact solutions of the original and modified conservation laws by measuring the $L^2$ error of the density component. The discretizations are chosen such that $\dx = \dy = 4\dt$ on all grids. Two different strategies for choosing the pseudo-time steps are considered:
\begin{description}
\item[Strategy 1:] Use $N=9$ iterations with $\dtau_{0,\dots,8} = 0.2 \dt$.
\item[Strategy 2:] Use $N=5$ iterations with $\dtau_0 = \dt$ and $\dtau_{1,\dots,4} = 0.2 \dt$.
\end{description}
Both strategies integrate in pseudo-time to the point $\tau = 1.8 \dt$ in each time step. However, Theorem \ref{thm:LW} predicts that {\bf Strategy 1} will give a speed modification $c(\mu_0, \dots, \mu_8) \approx 0.866$. On the other hand, {\bf Strategy 2} will give $c(\mu_0, \dots, \mu_4) = 1$, i.e. the correct propagation speed, due to the large initial pseudo-time step. Fig. \ref{fig:Euler_speed_comparison} shows the numerical solutions at time $T=10$ for the case where $\dx = \dy = 0.2$, $\dt = 0.05$. The predicted and observed vortex locations agree very well.

\begin{figure}[t!]
    \centering
    \includegraphics[width=\textwidth]{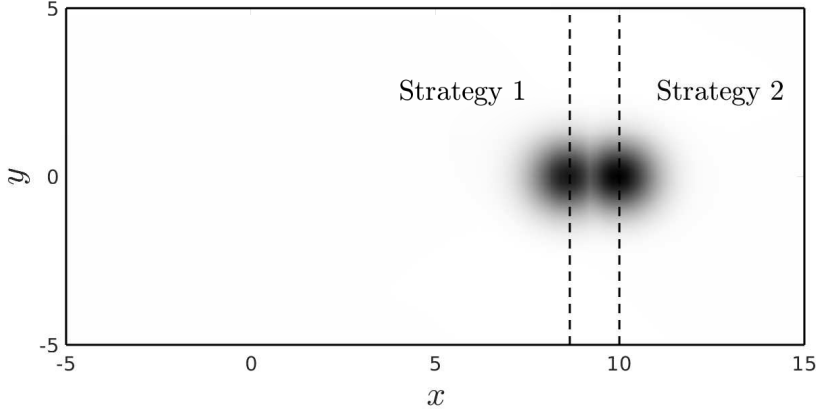}
    \caption{Computed density of the isentropic vortex at $t = 10$ using {\bf Strategy 1} and {\bf Strategy 2}. Dashed lines mark the vortex locations predicted by Theorem \ref{thm:LW}.}
    \label{fig:Euler_speed_comparison}
\end{figure}

Fig. \ref{fig:convergence} shows the convergence of the numerical solutions, measured at time $T=1$. Convergence to the correct solution is observed for {\bf Strategy 2} (S2) but not for {\bf Strategy 1} (S1). This is expected due to the incorrect location of the vortex in the latter case. However, convergence to the solution of the modified conservation law is seen (S1 mod). Thus, Theorem \ref{thm:LW} accurately predicts the behavior of the numerical solution despite the violated assumptions.

In practical applications, the pseudo-time iterations are terminated when the residual has decreased beneath some tolerance. It is interestig to see in what way the convergence is affected by the choice of strategy. Returning to the setting in Fig. \ref{fig:Euler_speed_comparison}, the residual in each pseudo-time iteration for all 200 physical time steps are shown for the two strategies in Fig. \ref{fig:residuals}. The 200 lines overlap nearly perfectly, suggesting that the residual behaves similarly in each physical time step. The initial iteration in {\bf Strategy 2} evidently has a large impact on the reduction of the residual that supercedes those of the other iterations put together. The remaining iterations appear to reduce the residual by comparable amounts for the two strategies, as seen by the similar gradients. In contrast to the shock problem considered previously, this suggests that the convergence rate of the residual is dictated by other factors than the propagation speed for this particular problem. Nonetheless, a correct propagation speed is visibly very beneficial, here with a drop in relative residual of more than an order of magnitude.

\begin{figure}[t!]
    \centering
    \begin{subfigure}{0.49\textwidth}
    \includegraphics[width=\textwidth]{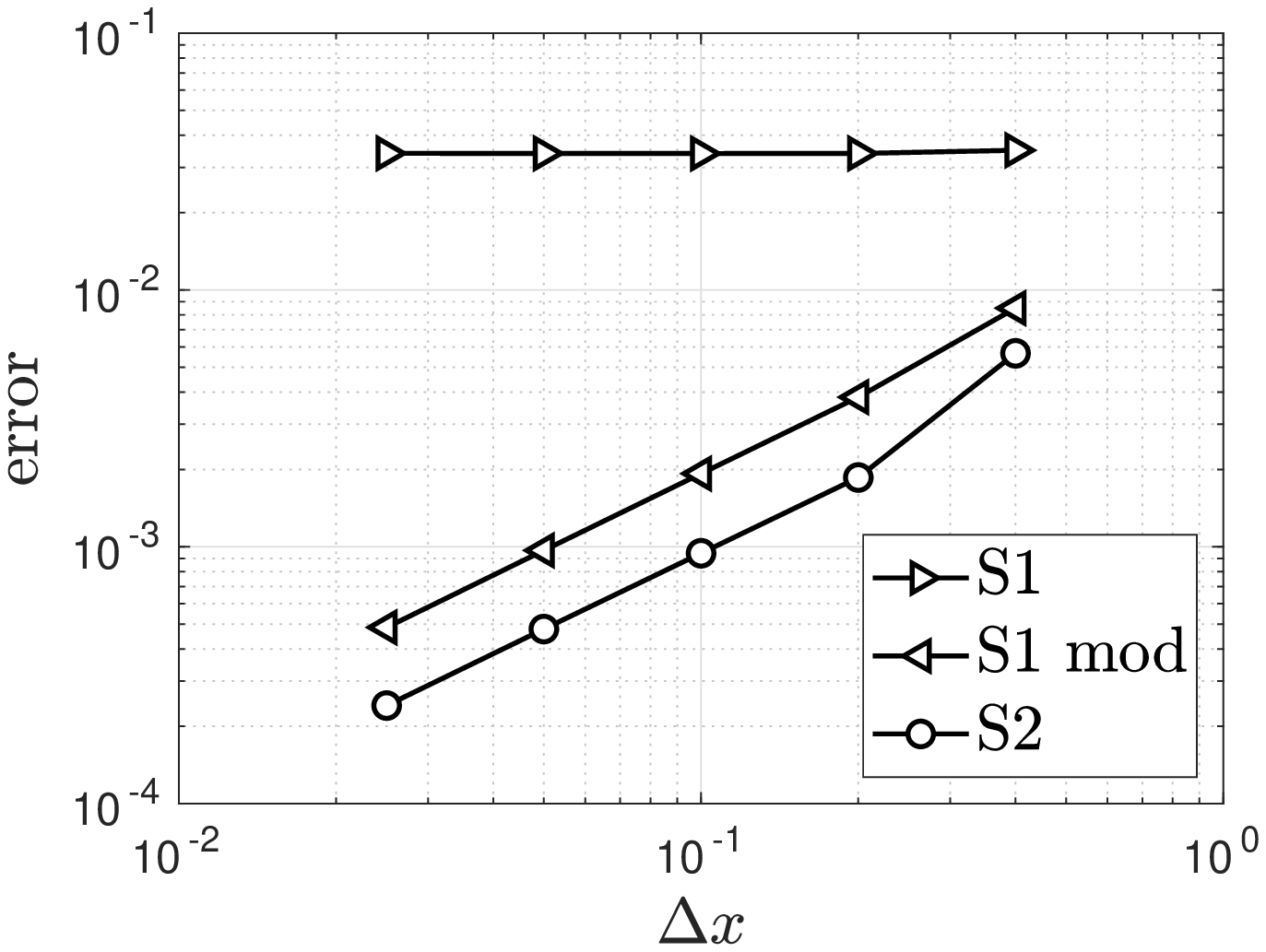}
    \caption{$L^2$-errors.}
    \label{fig:convergence}
    \end{subfigure}
    \begin{subfigure}{0.49\textwidth}
    \includegraphics[width=\textwidth]{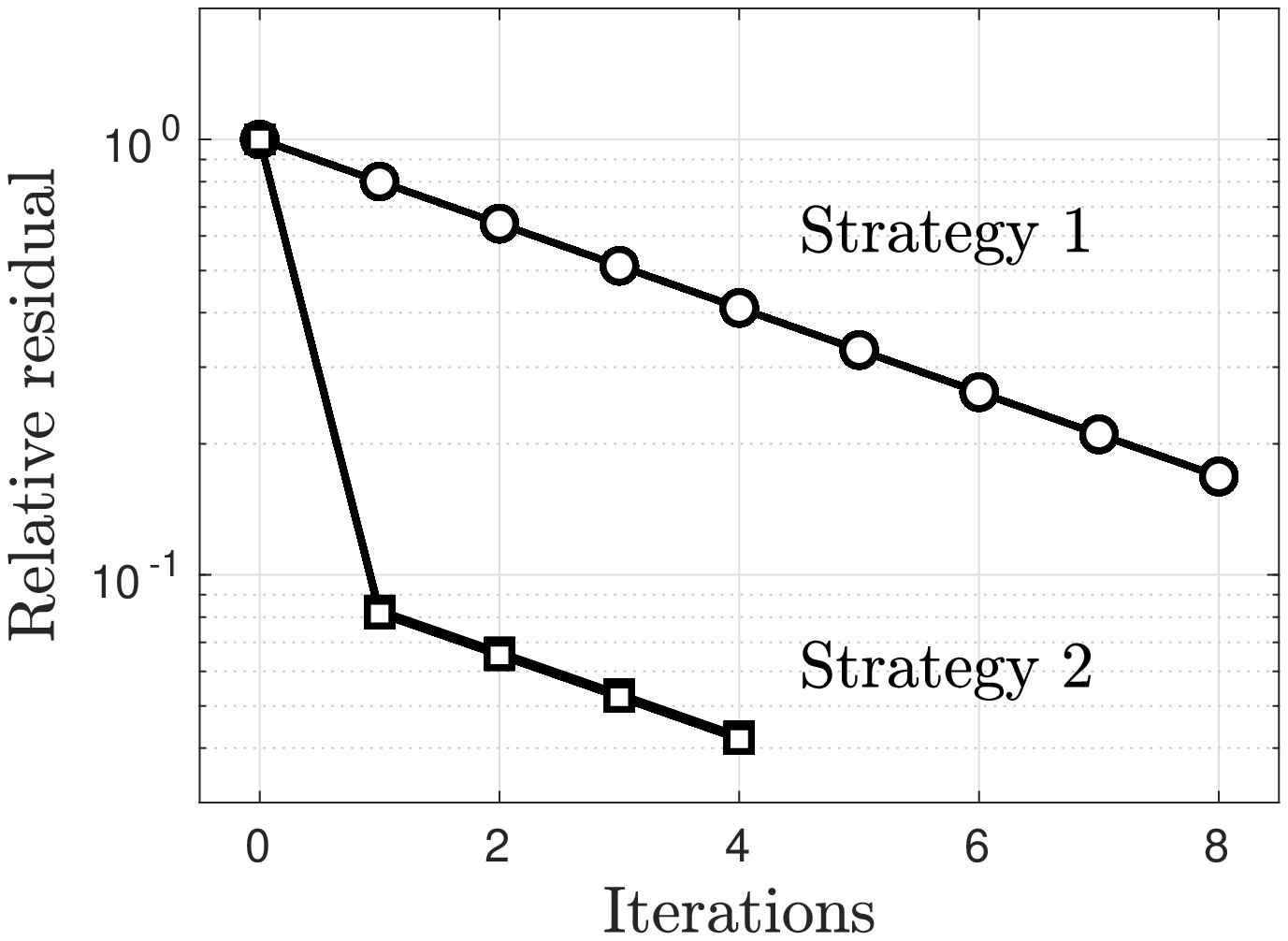}
    \caption{Relative residuals.}
    \label{fig:residuals}
    \end{subfigure}
    \caption{(a) $L^2$-error computed with respect to the exact solution using {\bf Strategy 1} (S1) and {\bf Strategy 2} (S2), and with respect to the modified conservation law (S1 mod). (b) Relative residuals at each iteration for 200 physical time steps.}
    \label{fig:Euler_error_and_residual}
\end{figure}


\section{Summary and Conclusions} \label{sec:conclusions}

In this paper we have studied conservation properties of a selection of iterative methods applied to 1D scalar conservation laws. The fact that conservation is a design principle behind many numerical schemes motivates such a study. We have established that Newton's method, the Richardson iteration, Krylov subspace methods, coarse grid corrections using agglomeration, as well as ERK pseudo-time iterations preserve the global conservation of a given scheme if the initial guess has correct mass. However, the Jacobi and Gauss-Seidel methods do not, unless the linear system being solved possesses particular properties.

The stronger requirement of local conservation has been investigated for ERK pseudo-time iterations. We have shown that local conservation is preserved for finite volume schemes that employ the implicit Euler method in time. However, the resulting modified numerical flux may be inconsistent with the governing conservation law. An extension of the Lax-Wendroff theorem shows that this inconsistency leads to convergence to a weak solution of a conservation law modified by a particular constant. We give an exact expression for the modification constant, which depends only on the stability function of the ERK method and the pseudo-time steps. Depending on the problem solved, this modification can alter both the propagation speed and amplitude of the numerical solution if the constant differs from unity. We present a strategy for ensuring that the constant equals one and show numerically that the strategy results in faster convergence. Experiments suggest that the results hold even for systems of conservation laws in multiple dimensions.


\bibliography{Manuscript.bib}


\appendix

\section{Proof of Lemma \ref{lemma:flux_form}} \label{appendix:proof}

The purpose of this appendix is to provide a detailed proof of Lemma \ref{lemma:flux_form}.

\begin{proof}
Note first that \eqref{eq:RK_step}-\eqref{eq:RK_stages} can be equivalently written in the form
\begin{equation} \label{eq:RK_matrix_vector_form}
    \begin{aligned}
        \tvec{U}_i^{(k)} &= \iter{k} \tvec{1} - \dtau_k \tmat{A} \tvec{g}_i(\tvec{\xvec{U}}^{(k)}), \\
        \iter{k+1} &= \iter{k} - \dtau_k \tvec{b}^\top \tvec{g}_i(\tvec{\xvec{U}}^{(k)}),
    \end{aligned}
    \qquad i = \dots, -1, 0, 1, \dots, 
\end{equation}
where $\tvec{1} = (1,\dots,1)^\top \in \mathbb{R}^s$. Here, we use the notation $\tvec{g}_i(\tvec{\xvec{U}}^k) = \left( g_i(\xvec{U}_1^{(k)}), \dots, g_i(\xvec{U}_s^{(k)}) \right)^\top$.

Suppose that for some $N \geq 1$ the relation
$$
\frac{\iter{N} - u_i^n}{\dt} + \frac{1}{\dx} \left( \hat{h}_{i+\frac{1}{2}}^{(N)} - \hat{h}_{i-\frac{1}{2}}^{(N)} \right) = 0
$$
holds. We will show that it also holds for $N+1$. From \eqref{eq:g-fun}, \eqref{eq:pseudo-time_conservative} and \eqref{eq:RK_matrix_vector_form} it follows that
\begin{equation*}
    \begin{alignedat}{3}
    \tvec{U}_i^{(N)} &= \iter{N} \tvec{1} &&- \dtau_N \tmat{A} \left[ \frac{\tvec{U}_i^{(N)} - u_i^n \tvec{1}}{\dt} + \frac{1}{\dx} \left( \hat{\tvec{f}}_{i+\frac{1}{2}}^{(N)} - \hat{\tvec{f}}_{i-\frac{1}{2}}^{(N)} \right) \right] \\
    &= u_i^n \tvec{1} &&- \frac{\dt}{\dx} \left( \hat{h}_{i+\frac{1}{2}}^{(N)} - \hat{h}_{i-\frac{1}{2}}^{(N)} \right) \tvec{1} \\
    &&&- \mu_N \tmat{A} \tvec{U}_i^{(N)} + \mu_N \tmat{A} \tvec{1} u_i^n - \frac{\dtau_N}{\dx} \tmat{A} \left( \hat{\tvec{f}}_{i+\frac{1}{2}}^{(N)} - \hat{\tvec{f}}_{i-\frac{1}{2}}^{(N)} \right).
    \end{alignedat}
\end{equation*}
Solving for $\tvec{U}_i^{(N)}$ gives
\begin{equation*}
    \begin{aligned}
    \tvec{U}_i^{(N)} = u_i^n \tvec{1} &- \frac{\dt}{\dx} \left( \hat{h}_{i+\frac{1}{2}}^{(N)} - \hat{h}_{i-\frac{1}{2}}^{(N)} \right) (\tmat{I} + \mu_N \tmat{A})^{-1} \tvec{1} \\
    &- \frac{\dt}{\dx} \mu_N (\tmat{I} + \mu_N \tmat{A})^{-1} \tmat{A} \left( \hat{\tvec{f}}_{i+\frac{1}{2}}^{(N)} - \hat{\tvec{f}}_{i-\frac{1}{2}}^{(N)} \right).
    \end{aligned}
\end{equation*}
Note that $(\tmat{I} + \mu_N \tmat{A})^{-1}$ exists since $\tmat{A}$ is lower triangular.

Evaluating $\tvec{g}_i \left( \tvec{\xvec{U}}^{(N)} \right)$ using \eqref{eq:g-fun} gives
\begin{equation*}
    \begin{aligned}
    \tvec{g}_i \left( \tvec{\xvec{U}}^{(N)} \right) &= \frac{1}{\dx} \left[-\left( \hat{h}_{i+\frac{1}{2}}^{(N)} - \hat{h}_{i-\frac{1}{2}}^{(N)} \right) (\tmat{I} + \mu_N \tmat{A})^{-1} \tvec{1} \right. \\
    & \qquad \quad \, \, \left. -\mu_N (\tmat{I} + \mu_N \tmat{A})^{-1} \tmat{A} \left( \hat{\tvec{f}}_{i+\frac{1}{2}}^{(N)} - \hat{\tvec{f}}_{i-\frac{1}{2}}^{(N)} \right) + \left( \hat{\tvec{f}}_{i+\frac{1}{2}}^{(N)} - \hat{\tvec{f}}_{i-\frac{1}{2}}^{(N)} \right) \right] \\
    &= \frac{1}{\dx} \left[ -\left( \hat{h}_{i+\frac{1}{2}}^{(N)} - \hat{h}_{i-\frac{1}{2}}^{(N)} \right) (\tmat{I} + \mu_N \tmat{A})^{-1} \tvec{1} \right. \\
    & \left. \qquad \quad \, \, +[\tmat{I} - \mu_N (\tmat{I} + \mu_N \tmat{A})^{-1}\tmat{A}] \left( \hat{\tvec{f}}_{i+\frac{1}{2}}^{(N)} - \hat{\tvec{f}}_{i-\frac{1}{2}}^{(N)} \right) \right] \\
    &= \frac{1}{\dx} (\tmat{I} + \mu_N \tmat{A})^{-1} \left[ \left( \hat{\tvec{f}}_{i+\frac{1}{2}}^{(N)} - \hat{\tvec{f}}_{i-\frac{1}{2}}^{(N)} \right) - \left( \hat{h}_{i+\frac{1}{2}}^{(N)} - \hat{h}_{i-\frac{1}{2}}^{(N)} \right) \tvec{1} \right].
    \end{aligned}
\end{equation*}
In the last equality we have used the fact that
\begin{equation} \label{eq:trick}
\tmat{I} - \mu_N (\tmat{I} + \mu_N \tmat{A})^{-1}\tmat{A} = (\tmat{I} + \mu_N \tmat{A})^{-1} [(\tmat{I} + \mu_N \tmat{A}) - \mu_N \tmat{A}] = (\tmat{I} + \mu_N \tmat{A})^{-1}.
\end{equation}
Inserting the above expression for $\tvec{g}_i \left( \tvec{\xvec{U}}^{(N)} \right)$ into $\iter{N+1}$ as given in \eqref{eq:RK_matrix_vector_form} leads to
\begin{equation*}
    \begin{alignedat}{3}
    \iter{N+1} &= \iter{N} &&- \dtau_N \tvec{b}^\top \tvec{g}_i \left( \tvec{\xvec{U}}^{(N)} \right) \\
    &= u_i^n &&- \frac{\dt}{\dx} \left( \hat{h}_{i+\frac{1}{2}}^{(N)} - \hat{h}_{i-\frac{1}{2}}^{(N)} \right) \\
    &&& -\frac{\dtau_N}{\dx} \tvec{b}^\top (\tmat{I} + \mu_N \tmat{A})^{-1} \left[ \left( \hat{\tvec{f}}_{i+\frac{1}{2}}^{(N)} - \hat{\tvec{f}}_{i-\frac{1}{2}}^{(N)} \right) - \left( \hat{h}_{i+\frac{1}{2}}^{(N)} - \hat{h}_{i-\frac{1}{2}}^{(N)} \right) \tvec{1} \right] \\
    &= u_i^n &&- \frac{\dt}{\dx} [1 - \mu_N \tvec{b}^\top (\tmat{I} + \mu_N \tmat{A})^{-1} \tvec{1}] \left( \hat{h}_{i+\frac{1}{2}}^{(N)} - \hat{h}_{i-\frac{1}{2}}^{(N)} \right) \\
    &&& -\frac{\dt}{\dx} \mu_N \tvec{b}^\top (\tmat{I} + \mu_N \tmat{A})^{-1} \left( \hat{\tvec{f}}_{i+\frac{1}{2}}^{(N)} - \hat{\tvec{f}}_{i-\frac{1}{2}}^{(N)} \right) \\
    &= u_i^n &&- \frac{\dt}{\dx} \left[ \phi(-\mu_N) \left( \hat{h}_{i+\frac{1}{2}}^{(N)} - \hat{h}_{i-\frac{1}{2}}^{(N)} \right) + \mu_N \tvec{b}^\top (\tmat{I} + \mu_N \tmat{A})^{-1} \left( \hat{\tvec{f}}_{i+\frac{1}{2}}^{(N)} - \hat{\tvec{f}}_{i-\frac{1}{2}}^{(N)} \right) \right].
    \end{alignedat}
\end{equation*}
Rearranging, using the induction hypothesis and the expression \eqref{eq:H-flux} for $\hat{h}^N$ results in
\begin{equation*}
    \begin{aligned}
    0 &= \frac{\iter{N+1} - u_i^n}{\dt} \\
    &+ \frac{1}{\dx} \left[ \phi(-\mu_N) \sum_{k=0}^{N-1} \mu_k \tvec{b}^\top (\tmat{I} + \mu_k \tmat{A})^{-1} \left( \prod_{l=k+1}^{N-1} \phi(-\mu_{l}) \right) \left( \hat{\tvec{f}}_{i+\frac{1}{2}}^{(k)} - \hat{\tvec{f}}_{i-\frac{1}{2}}^{(k)} \right) \right. \\
    &+ \left. \mu_N \tvec{b}^\top (\tmat{I} + \mu_N \tmat{A})^{-1} \left( \hat{\tvec{f}}_{i+\frac{1}{2}}^{(N)} - \hat{\tvec{f}}_{i-\frac{1}{2}}^{(N)} \right) \right] \\
    &= \frac{\iter{N+1} - u_i^n}{\dt} + \frac{1}{\dx} \sum_{k=0}^N \mu_k \tvec{b}^\top (\tmat{I} + \mu_k \tmat{A})^{-1} \left( \prod_{l=k+1}^N \phi(-\mu_{l}) \right) \left( \hat{\tvec{f}}_{i+\frac{1}{2}}^{(k)} - \hat{\tvec{f}}_{i-\frac{1}{2}}^{(k)} \right) \\
    &= \frac{\iter{N+1} - u_i^n}{\dt} + \frac{1}{\dx} \left( \hat{h}_{i+\frac{1}{2}}^{(N+1)} - \hat{h}_{i-\frac{1}{2}}^{(N+1)} \right).
    \end{aligned}
\end{equation*}

It remains to show that the lemma holds when $N=1$. To this end, recall that $\iter{0} = u_i^n$ and note from \eqref{eq:g-fun} and \eqref{eq:RK_matrix_vector_form} that
\begin{equation*}
    \begin{aligned}
    \tvec{U}_i^{(0)} &= u_i^n \tvec{1} - \dtau_0 \tmat{A} \left[ \frac{\tvec{U}_i^{(0)} - u_i^n \tvec{1}}{\dt} + \frac{1}{\dx} \left( \hat{\tvec{f}}_{i+\frac{1}{2}}^{(0)} - \hat{\tvec{f}}_{i-\frac{1}{2}}^{(0)} \right) \right] \\
    &= u_i^n \tvec{1} - \mu_0 \tmat{A} \tvec{U}_i^{(0)} + \mu_0 \tmat{A} \tvec{1} u_i^n - \frac{\dtau_0}{\dx} \tmat{A} \left( \hat{\tvec{f}}_{i+\frac{1}{2}}^{(0)} - \hat{\tvec{f}}_{i-\frac{1}{2}}^{(0)} \right).
    \end{aligned}
\end{equation*}
Solving for $\tvec{U}_i^{(0)}$ gives
$$
\tvec{U}_i^{(0)} = u_i^n \tvec{1} - \frac{\dtau_0}{\dx} (\tmat{I} + \mu_0 \tmat{A})^{-1} \tmat{A} \left( \hat{\tvec{f}}_{i+\frac{1}{2}}^{(0)} - \hat{\tvec{f}}_{i-\frac{1}{2}}^{(0)} \right).
$$
Using \eqref{eq:g-fun} it follows that
\begin{equation*}
    \begin{aligned}
    \tvec{g} \left( \tvec{\xvec{U}}^{(0)} \right) &= \frac{1}{\dx} \left[ -\mu_0 (\tmat{I} + \mu_0 \tmat{A})^{-1} \tmat{A} \left( \hat{\tvec{f}}_{i+\frac{1}{2}}^{(0)} - \hat{\tvec{f}}_{i-\frac{1}{2}}^{(0)} \right) + \left( \hat{\tvec{f}}_{i+\frac{1}{2}}^{(0)} - \hat{\tvec{f}}_{i-\frac{1}{2}}^{(0)} \right) \right] \\
    &= \frac{1}{\dx} [\tmat{I} -\mu_0 (\tmat{I} + \mu_0 \tmat{A})^{-1} \tmat{A}] \left( \hat{\tvec{f}}_{i+\frac{1}{2}}^{(0)} - \hat{\tvec{f}}_{i-\frac{1}{2}}^{(0)} \right) \\
    &= \frac{1}{\dx} (\tmat{I} + \mu_0 \tmat{A})^{-1} \left( \hat{\tvec{f}}_{i+\frac{1}{2}}^{(0)} - \hat{\tvec{f}}_{i-\frac{1}{2}}^{(0)} \right).
    \end{aligned}
\end{equation*}
Here we have once again used \eqref{eq:trick} in the final equality. Thus, $\iter{1}$ can be evaluated using \eqref{eq:RK_matrix_vector_form} as
$$
\iter{1} = u_i^n - \frac{\dtau_0}{\dx} \tvec{b}^\top (\tmat{I} + \mu_0 \tmat{A})^{-1} \left( \hat{\tvec{f}}_{i+\frac{1}{2}}^{(0)} - \hat{\tvec{f}}_{i-\frac{1}{2}}^{(0)} \right) = u_i^n - \frac{\dt}{\dx} \left( \hat{h}_{i+\frac{1}{2}}^{(1)} - \hat{h}_{i-\frac{1}{2}}^{(1)} \right).
$$
Division by $\dt$ and rearrangement shows that the lemma holds when $N=1$. By induction it holds for all $N \geq 1$.
\end{proof}


\end{document}